\documentclass[11pt,a4paper,final,abstracton,numbers=noenddot]{scrartcl}

\usepackage[utf8]{inputenc}

\usepackage{titling}
\usepackage{graphicx}
\usepackage[english]{babel}
\usepackage{datetime}
\usepackage{enumitem}
\newlist{ienum}{enumerate}{3}
\setlist[enumerate]{label=(\roman*),partopsep=0pt,itemsep=0pt,topsep=3pt,wide,leftmargin=0.2em,labelindent=0em}
\setlist[ienum]{label=(\roman*),partopsep=0pt,itemsep=0pt,topsep=3pt,wide,leftmargin=.5em,labelindent=.5em}
\setlist{nolistsep}
\usepackage{tikz}
\usetikzlibrary{matrix,arrows}
\usetikzlibrary{matrix,positioning}
\usetikzlibrary{decorations.markings}
\usetikzlibrary{calc}
\usetikzlibrary{shapes}
\usepackage{pgfplots}
    \pgfplotsset{
        compat=1.12,
    }
\usetikzlibrary{intersections}

\usepackage{fancyhdr}
\usepackage[style=alphabetic,citestyle=alphabetic]{biblatex}
\DeclareFieldFormat[article]{title}{\mkbibemph{#1\isdot}}
\defbibheading{bibliography}{\centering\huge\textbf{References}\vspace{1cm}}

\usepackage[pdftex,draft=false]{hyperref}
\hypersetup{%
pdftitle = {   },	
pdfsubject = {  },	
pdfauthor = { Stefan Schmid },
pdfkeywords = { number theory, elliptic curves, singular moduli },
plainpages=true,
pdfborder=0 0 0
}

\usepackage{amssymb}
\usepackage{amsmath}
\usepackage{amsthm}

\newtheoremstyle{named}
{}
{}
{\itshape}
{}
{\bfseries}
{}
{\newline}
{\thmname{#1} \thmname{#2} (\textit{\thmname{#3}})}

\newtheoremstyle{mythm}
{}
{}
{\itshape}
{}
{\bfseries}
{}
{1em}
{\thmname{#1} \thmname{#2}}

\newtheoremstyle{futref}
{}
{}
{\itshape}
{}
{\bfseries}
{}
{1em}
{\thmname{#1} \thmname{#3}}

\newtheoremstyle{mydef}
{}
{}
{\itshape\let\emph\textbf}
{}
{\bfseries}
{.}
{1em}
{\thmname{#1}}

\newtheoremstyle{myex}
{}
{}
{}
{}
{\bfseries}
{}
{1em}
{\thmname{#1} \thmname{#2}}

\theoremstyle{futref}

\theoremstyle{mythm}
\newtheorem{theorem}{Theorem}
\newtheorem{prop}[theorem]{Proposition}
\newtheorem{lem}[theorem]{Lemma}
\newtheorem{cor}[theorem]{Corollary}
\theoremstyle{mydef}

\theoremstyle{myex}

\theoremstyle{named}

\theoremstyle{remark}

\let\oldproof=\proof
\let\oldendproof=\endproof
\renewenvironment{proof}{\oldproof[\textsc{\proofname}]}{\oldendproof}

\newcommand{\N}{\ensuremath{\mathbb{N}}}    
\newcommand{\Z}{\ensuremath{\mathbb{Z}}}    
\newcommand{\Q}{\ensuremath{\mathbb{Q}}}    
\renewcommand{\H}{\ensuremath{\mathbb{H}}}    
\newcommand{\C}{\ensuremath{\mathbb{C}}}    
\newcommand{\F}{\ensuremath{\mathcal{F}}}	
\newcommand{\Pen}{\ensuremath{\mathcal{P}}}	

\newcommand{\M}{\ensuremath{\mathcal{M}}}		

\newenvironment{smatrix}{\bigl(\begin{smallmatrix}}{\end{smallmatrix}\bigr)}

\let\Re\relax
\DeclareMathOperator{\Re}{\operatorname{Re}}
\let\Im\relax
\DeclareMathOperator{\Im}{\operatorname{Im}}

\DeclareMathOperator{\GL}{\operatorname{GL}}
\DeclareMathOperator{\SL}{\operatorname{SL}}

\DeclareMathOperator{\Mat}{\operatorname{Mat}}
\DeclareMathOperator{\sl2z}{\operatorname{SL_2(\Z)}}

\DeclareMathOperator{\vol}{vol}

\DeclareMathOperator{\Aut}{Aut}
\DeclareMathOperator{\Gal}{Gal}

\DeclareMathOperator{\im}{im}

\newcommand{\twostack}[2]{\subarray{c}\scriptscriptstyle#1\\\scriptscriptstyle#2\endsubarray}

\let\emph\textit
\let\subset\subseteq

\numberwithin{theorem}{section}

\title{Integrality properties in the Moduli Space of Elliptic Curves: Isogeny Case}
\author{Stefan Schmid}
\date{}

\begin{document}

\setlist[enumerate]{label=(\roman*),partopsep=0pt,itemsep=0pt,topsep=3pt,wide,leftmargin=0.2em,labelindent=0em}
\setlist[ienum]{label=(\roman*),partopsep=0pt,itemsep=0pt,topsep=3pt,wide,leftmargin=.5em,labelindent=.5em}

\maketitle

\begin{abstract}
	For a fixed $j$--invariant $j_0$ of an elliptic curve without complex multiplication
	we bound the number of $j$--invariants $j$ that are algebraic units and such that
	elliptic curves corresponding to $j$ and $j_0$ are isogenous.
	Our bounds are effective. We also modify the problem slightly by fixing a
	singular modulus $\alpha$ and looking at all $j$ such that $j-\alpha$ is an
	algebraic unit and such that elliptic curves corresponding to $j$ and $j_0$
	are isogenous. The number of such $j$ can again be bounded effectively.
\end{abstract}

\thispagestyle{empty}

\newcommand{\T}{\ensuremath{\mathcal{T}}}

\section{Introduction}

In this text, $K$ will be a number field. By a \emph{finite place} $\nu$ of $K$
we mean a non--archimedean absolute value on $K$ that restricts to the
$p$--adic absolute value on $\Q$ for some rational prime $p$.
We thus have $|p|_\nu = p^{-1}$.
The integer $d_v$ will denote the degree of the completion of $K$ with respect to
the valuation $\nu$ over the field $\Q_p$. We define the (absolute logarithmic) height
of an algebraic number $j$ by
\begin{displaymath}
	h(j) = \frac 1{[K:\Q]} \left( \sum_\sigma \log \max \{1,|\sigma(j)|\}
				+ \sum_\nu d_\nu \log \max\{1,|j|_\nu\} \right),
\end{displaymath}
where $K$ is any number field containing $j$ and $\sigma$ runs over all field
embeddings $\sigma\colon K \rightarrow \C$ and $\nu$ runs over all finite places of $K$.
This definition is independent of the choice of $K$.
Note that $j$ is an algebraic integer if and only if $|j|_\nu \le 1$ holds for
all finite places. 
Thus the height of an algebraic integer is given by
\begin{displaymath}
	h(j) = \frac 1{[K:\Q]} \sum_{|\sigma(j)|>1} \log |\sigma(j)|.
\end{displaymath}
Since $h(j) = h(j^{-1})$ holds for all algebraic numbers, we obtain for
an algebraic integer the equality
\begin{align*}
	h(j) &= \frac 1{[K:\Q]} \sum_{|\sigma(j)|>1} \log |\sigma(j)|	\notag	\\
	&= -\frac 1{[K:\Q]} \left( \sum_{|\sigma(j)|<1} \log |\sigma(j)|
				+ \sum_\nu d_\nu \log |j|_\nu \right).
\end{align*}

Let $j$ be an algebraic unit, so that it has norm $\pm 1$. The height of $j$ then reduces to
\begin{equation}\label{eq:alg_int_height}
	h(j) = \frac 1{\left[K:\Q\right]} \sum_{\lvert \sigma(j) \rvert > 1} \log \lvert \sigma(j) \rvert
		= -\frac 1{\left[K:\Q\right]} \sum_{\lvert \sigma(j) \rvert < 1} \log \lvert \sigma(j) \rvert
\end{equation}
where $D = \left[K:\Q\right]$, and $\sigma$ runs over all $\Q$--homomorphisms $K \hookrightarrow \C$.

Note that $h$ also denotes the Faltings height (with the $\frac 12\log\pi$--term) but
there should be no ambiguity.

The multiplicative height will be denoted by $H(\cdot) = e^{h(\cdot)}$ and satisfies
\begin{equation}\label{eq:ht_sum_lower_bound}
	H(\alpha\beta) \le 2H(\alpha)H(\beta)
\end{equation}
for any two algebraic numbers $\alpha$ and $\beta$.

We will denote Klein's modular function by $j\colon \H \rightarrow \C$.
For a fixed $\alpha \in \bar\Q$ the number of $j$--invariants $j$ of elliptic curves
with complex multiplication such that $j-\alpha$ is a unit can be effectively bounded.
See \cite{bhk} and \cite{cmcase} for details.

We want to look at a similar problem where $j$ does not have complex multiplication.
Without further assumptions the number of such $j$ can not be bounded.
We thus fix an elliptic curve without complex multiplication, and denote by $j_0$ its
$j$--invariant. Assume that the curve is defined over a number field $K$ contained
in $\C$.
Our aim is to prove the following result.
\begin{theorem}
	
	Let $j_0$ be the $j$--invariant of an elliptic curve without complex multiplication.
	Then there are at most finitely many $j$--invariants $j$ of elliptic curves that are isogenous to
	an elliptic curve corresponding to $j_0$ and such that $j$ is an algebraic unit.

\end{theorem}

To be precise, we give an effective bound for the degree of the minimal isogeny. This leaves
only finitely many possibilities for $j$.

We can also look a slightly different problem and fix the $j$--invariant $\alpha$ of an elliptic
curve with complex multiplication. We have the same result as stated in the following theorem.
\begin{theorem}
	
	Assume $\alpha$ is the $j$--invariant of an elliptic curve with CM.
	Let $j_0$ be the $j$--invariant of an elliptic curve without CM.
	Then there are at most finitely many $j$--invariants $j$ of elliptic curves that are isogenous to
	an elliptic curve corresponding to $j_0$ and such that $j-\alpha$ is an algebraic unit.

\end{theorem}

Again we give a bound on the degree of the minimal isogeny and our bounds are effective.

\section{Isogenous points in the fundamental domain}

Recall that $j_0$ is the $j$--invariant of a fixed elliptic curve without CM.
Assume $j(\tau_0) = j_0$ for $\tau_0 \in \F$.
For any number field $K$ and any embedding $\sigma\colon K \hookrightarrow \C$ there is a $\tau_0^\sigma \in \F$
such that $j(\tau_0^\sigma) = \sigma(j_0)$.

For $\xi \in \bar\F$ and $\tau \in \H$
we define the sets
\begin{displaymath}
	\Sigma(\xi, \varepsilon) = \left\{ \tau \in \F; \lvert j(\tau) - j(\xi) \rvert < \varepsilon \right\}
\end{displaymath}
and
\begin{displaymath}
	\Gamma(\xi, \varepsilon) = \left\{ \sigma\colon K \rightarrow \C; \tau_0^\sigma \in \Sigma(\xi,\varepsilon) \right\}.
\end{displaymath}
We will write $\Sigma_\varepsilon$ and $\Gamma_\varepsilon$ for $\Sigma(\zeta,\varepsilon)$ and $\Gamma(\zeta,\varepsilon)$,
respectively, where $\zeta = e^{2\pi i/6}$.

We want to give an explicit bound for the number of elements in the Galois orbit of $j_0$ satisfying the
condition above. First, we will bound the number of points in the Hecke orbit, and then use a
result of Lombardo to estimate the total number. Two (equivalence classes of isomorphic) elliptic
curves are in the same Hecke orbit if they are isogenous.

We will need the following counting lemma. We translate points in the upper
half--plane into the fundamental domain with matrices in $\sl2z$, and thus get restrictions
on then entries of the matrices.

\begin{lem}\label{lem:ellipse}
	Let $\xi \in \bar\F$ and $\varepsilon \in (0,\frac{\sqrt{3}}{3|\xi|+2}]$. Let $\tau \in \H$
	satisfy $|\tilde\tau - \xi| \le \varepsilon$, where $\tilde\tau \in \F$ is in the $\sl2z$--orbit of $\tau$.
	Pick
	\begin{displaymath}
		\gamma = \begin{pmatrix}a&b\\c&d\end{pmatrix} \in \sl2z
	\end{displaymath}
	such that $\gamma\tau = \tilde\tau$.
	Then there exist $\nu \in \{\pm 1\}$ such that
	\begin{equation}\label{eq:matrixentries}
		\left\lvert a^2 + \nu 2 |\Re(\xi)| ac + |\xi|^2 c^2 - \frac{\Im(\xi)}{\Im(\tau)}\right\rvert
		\le 7\frac{4|\xi|+1}{\sqrt 3} |\xi|^2 \frac{\varepsilon^{1/2}}{\Im(\tau)},
	\end{equation}
	and
	\begin{equation}\label{eq:bound_a2_c2}
		\max\left\{a^2,c^2\right\} \le \frac{4|\xi|+1}{\sqrt 3} \frac 1{\Im(\tau)}.
	\end{equation}
	Moreover, we have
	\begin{displaymath}
		|d| \le |c| |\Re(\tau)| + \frac{4|\xi|+1}{\sqrt{3}}
	\end{displaymath}
	and
	\begin{displaymath}
		|b| \le |a| |\Re(\tau)| + \frac{4|\xi|+1}{\sqrt 3}.
	\end{displaymath}
\end{lem}

\noindent The lemma tells us, that the first column of $\gamma$, considered as a point in the
plane, is close to a conic section.
Since
\begin{displaymath}
	(2\nu |\Re(\xi)|)^2 - 4|\xi|^2 = 4(\Re(\xi)^2 - \Re(\xi)^2- \Im(\xi)^2) = -4\Im(\xi) < 0
\end{displaymath}
the equation actually defines an ellipse.
The ellipse is defined in terms of $\xi$ and $\tau$.

\begin{proof}
	Let $R = |\xi|$ and $A = \Im(\xi)$.
	Moreover write $\tau = x + iy$. We have
	\begin{displaymath}
		\Im(\gamma\tau) = \frac{\Im\tau}{(cx+d)^2+c^2y^2} \ge A - \varepsilon
	 	\ge A - \frac{\sqrt{3}}{3R+2} \ge \frac{\sqrt{3}}{4R+1}
	\end{displaymath}
	by definition of $\varepsilon$ and $A \ge \frac{\sqrt{3}}2$.
	Define $\delta_1 := \Im(\gamma\tau)^{-1}$. Then $\delta_1 \le 1/(A-\varepsilon)$ and
	\begin{equation}\label{eq:cddependence}
		(cx+d)^2 + c^2y^2 = \delta_1 y.
	\end{equation}
	This yields $c^2 \le \delta_1/y \le (A-\varepsilon)^{-1}y^{-1}$, which implies the bound on $c^2$, and
	\begin{displaymath}
		|cx+d||c|y \le \frac 12 \left( (cx+d)^2 + c^2y^2\right) = \frac 12 \delta_1 y \le \frac 1{2(A-\varepsilon)} y.
	\end{displaymath}
	Further we get
	\begin{equation}\label{eq:dequation}
		|cx+d||c| \le \frac1{2(A-\varepsilon)} \le \frac{4R+1}{\sqrt{3}}
	\end{equation}
	and hence
	\begin{displaymath}
		|d| \le |c||\Re(\tau)| + \frac{4R+1}{\sqrt{3}}
	\end{displaymath}
	if $c \not= 0$.
	Thus the inequality for $d$ in the statement is true in the case $c \not= 0$. But if $c = 0$,
	then $d = \pm 1$, and $|d| = 1 \le \frac 5{\sqrt{3}} \le \frac{4R +1}{\sqrt{3}}$. Thus, the
	inequality for $d$ holds in both cases.

	Put $\gamma' = \begin{pmatrix}0&-1\\1&0\end{pmatrix}\gamma$.
	Then $\gamma'\tau = -\frac 1{\tilde\tau}$. Define $\delta_2 := \Im(\gamma'\tau)^{-1}$,
	i.e.
	\begin{equation}\label{eq:abdependence}
		(ax+b)^2 + a^2y^2 = \delta_2y.
	\end{equation}
	We put $r = |\tilde\tau|$ and $B = \Im(\tilde\tau)$.
	Now by the general rule of transformation of the imaginary part under fractional linear
	transformations
	\begin{displaymath}
		\delta_2 = \Im(\gamma'\tau)^{-1} = \Im\left(-\frac 1{\tilde\tau}\right)^{-1} 
		= \left( \frac{\Im(\tilde\tau)}{|\tilde\tau|^2} \right)^{-1} = \frac{r^2}B.
	\end{displaymath}
	We remark that $B/r = \Im(\tilde\tau/|\tilde\tau|) \ge \sqrt{3}/2$ since
	$\tilde\tau/r \in \bar\F$, and similarly $A/R \ge \sqrt{3}/2$.
	This implies
	\begin{displaymath}
		\delta_2 \le \frac 2{\sqrt{3}} r \le \frac 2{\sqrt{3}}(R + \varepsilon) \le \frac{2R+1}{\sqrt{3}}.
	\end{displaymath}
	We proceed as before with the bound on $d$ and $c^2$. From \eqref{eq:abdependence}
	we obtain $a^2 \le \delta_2/y \le (3R+1)/(\sqrt{3}y)$,
	which is the desired inequality of the statement. Moreover, we obtain
	\begin{equation}\label{eq:bequation}
		|ax+b||a| \le \delta_2/2 \le (R+1)/\sqrt{3}
	\end{equation}
	and hence
	\begin{displaymath}
		|b| \le |a||\Re(\tau)| + \frac{R+1}{\sqrt{3}},
	\end{displaymath}
	whenever $a \not= 0$. Again, if $a=0$, then $|b| = 1 \le \frac{4R+1}{\sqrt{3}}$, as claimed.

	It remains to prove \eqref{eq:matrixentries}.
	We deal with the case $c=0$ first. Then $a=d=\pm 1$ and $y = \delta_1^{-1} = \Im\tilde\tau = \Im\tau$,
	and thus $|y-A| \le |\tilde\tau - \xi| \le \varepsilon$.
	This implies
	\begin{displaymath}
		\left\vert \frac 1A - \frac 1y \right\vert \le \frac \varepsilon{Ay} \le \frac 1A \frac{\varepsilon^{1/2}}y.
	\end{displaymath}
	Multiplying by $A$ shows that Equation \eqref{eq:matrixentries} is true for any value of $\nu$.
	Now assume $c \not= 0$.
	We want to prove
	$|\delta_2 - \frac{|\xi|^2}{\Im \xi}| = |\delta_2 - \frac{|\xi|^2}{A}| \ll \varepsilon$.
	We compute
	\begin{equation}\label{eq:inequalityfora0}
		\begin{split}
		\left\vert\delta_2 - \frac{R^2}A\right\vert &= \left\vert\frac{r^2}B - \frac{R^2}A\right\vert
		= \left\vert\frac{R^2B - r^2A}{AB}\right\vert	\\
		&= \left\vert\frac{R^2B - BRr + BRr - ARr + ARr - r^2A}{AB}\right\vert	\\
		&\le RB\left\vert\frac{r-R}{AB}\right\vert + Rr\left\vert\frac{A-B}{AB}\right\vert
		+ rA\left\vert\frac{r-R}{AB}\right\vert	\\
		&\le \frac 2{\sqrt 3}\varepsilon + \frac 43 \varepsilon + \frac 2{\sqrt 3}\varepsilon	\\
		&\le 4\varepsilon,
		\end{split}
	\end{equation}
	where we have used $|R - r| = ||\xi|-|\tilde\tau|| \le |\xi - \tilde\tau| \le \varepsilon$
	and $|A - B| = |\Im(\xi) - \Im(\tilde\tau)| \le |\xi - \tilde\tau| \le \varepsilon$ in the
	second last inequality.

	Suppose $a=0$ for now. Then $b=-c=\pm 1$ and $y = \delta_2^{-1}$. Multiplying \eqref{eq:inequalityfora0} by
	$\Im(\xi) = A$ shows \eqref{eq:matrixentries} as the following argument shows.
	We have $\delta_2^{-1} = \Im(\gamma'\tau) = \frac{\Im(\tau)}{|b|^2} = \Im(\tau)$ by the
	usual transformation formula for the imaginary part of the action of $\sl2z$ by fractional
	linear transformations.
	Thus
	\begin{align*}
		|A\delta_2 - R^2| = \left\vert R^2 - \frac A{\Im(\tau)}\right\vert
		= \left\vert |\xi|^2 - \frac{\Im(\xi)}{\Im(\tau)}\right\vert \le 4A\varepsilon \le 4R \varepsilon^{1/2}.
	\end{align*}
	We have $\Im(\tau) \le 2/\sqrt{3}$ since $a=0$ and $\gamma$ translates $\tau$ into the fundamental domain.
	Therefore, the inequality remains true after multiplying the right--hand side by $2/\sqrt{3}\Im(\tau)^{-1}$.
	This shows equation \eqref{eq:matrixentries}.

	Finally, assume $ac \not=0$.
	Put $X := x +d/c$ and $Y := x + b/a$. Consider the difference of the two
	\begin{displaymath}
		X - Y = \left( x + \frac dc \right) - \left( x + \frac ba \right) = \frac 1{ac}.
	\end{displaymath}
	If we divide \eqref{eq:cddependence} by $c^2$ and rewrite the result in terms of $Y$ we get
	\begin{displaymath}
		0 = X^2 + y^2 - \frac{\delta_1 y}{c^2} = \left( Y + \frac 1{ac} \right)^2 + y^2 - \frac{\delta_1 y}{c^2}
		= Y^2 + \frac 2{ac} Y + \frac 1{(ac)^2} + y^2 - \frac{\delta_1 y}{c^2}.
	\end{displaymath}
	Similarly, if we divide \eqref{eq:abdependence} by $a^2$ we find
	\begin{displaymath}
		0 = Y^2 + y^2 - \frac{\delta_2 y}{a^2}
	\end{displaymath}
	Computing the resultant of the last two displays as polynomials in $Y$, and multiplying the result by
	$(ac)^4$ to kill the denominators,  gives us the expression
	\begin{equation}\label{eq:resultant}
		a^4y^2\delta_1^2 - 2a^2c^2y^2\delta_1\delta_2 + c^4y^2\delta_2^2 + 4a^2c^2y^2 - 2a^2y\delta_1 - 2c^2y\delta_2 + 1 = 0.
	\end{equation}
	Now write $\delta_1 = \frac 1A + \varepsilon_1$ and $\delta_2 = \frac{R^2}A + \varepsilon_2$.
	Then
	\begin{displaymath}
		|\varepsilon_1| = \left\vert\delta_1 - \frac 1A\right\vert
		= \left\vert \frac{A-\Im(\tilde\tau)}{A\Im(\tilde\tau)} \right\vert \le \frac 2{\sqrt{3}A} \varepsilon
	\end{displaymath}
	since $\Im(\tilde\tau) \ge \sqrt{3}/2$ and $|\Im(\xi)-\Im(\tilde\tau)| \le |\xi-\tilde\tau| \le \varepsilon$.
	Also $|\varepsilon_2| \le 4\varepsilon$ by \eqref{eq:inequalityfora0}.
	Put $\sigma = \Re(\xi)$. If we substitute these expressions for $\delta_1$ and $\delta_2$ in
	\eqref{eq:resultant} we obtain
	\begin{equation}\label{eq:resultant_times_const}
		\begin{split}
			0 =& a^4y^2\left( \frac 1A + \varepsilon_1 \right)^2
				- 2a^2c^2y^2\left( \frac 1A + \varepsilon_1 \right)\left( \frac{R^2}A + \varepsilon_2 \right)	\\
			&+ c^4y^2\left( \frac{R^2}A + \varepsilon_2 \right)^2 + 4a^2c^2y^2
				- 2a^2y\left( \frac 1A + \varepsilon_1 \right) - 2c^2y\left( \frac{R^2}A + \varepsilon_2 \right) + 1.
		\end{split}
	\end{equation}
	After multiplying the equation by $A^2/y^2$ the terms that do not include $\varepsilon_1$ and $\varepsilon_2$
	are given by
	\begin{align*}
		a^4 - 2a^2c^2A \frac{R^2}A &+ c^4A^2 \frac{R^4}{A^2} + 4a^2c^2A^2
		- 2a^2\frac Ay - 2c^2\frac{R^2}A \frac{A^2}y + \frac{A^2}{y^2}	\\
		&= a^4 - 2a^2c^2 R^2 + c^4R^4 + 4a^2c^2A^2 - 2a^2\frac Ay - 2c^2R^2 \frac Ay + \frac{A^2}{y^2}	\\
		&= a^4 - 2a^2c^2 R^2 + c^4R^4 + 4a^2c^2(R^2 - \sigma^2) - 2a^2\frac Ay - 2c^2R^2 \frac Ay + \frac{A^2}{y^2}	\\
		&= \left( a^2 - 2\sigma ac + R^2c^2 - \frac Ay \right) \left( a^2 + 2\sigma ac + R^2c^2 - \frac Ay \right)
	\end{align*}
	The terms that involve $\varepsilon_1$ and $\varepsilon_2$ in  \eqref{eq:resultant_times_const} after
	multiplying it by $A^2/y^2$ are given by
	\begin{align*}
		A^2 \left( \left(a^4 \frac 2A - 2a^2c^2 \frac{R^2}A - 2a^2 \frac 1y \right) \varepsilon_1\right.
		&+ \left(-2a^2c^2\frac 1A + 2c^4 \frac{R^2}A - 2c^2 \frac 1y\right) \varepsilon_2	\\
		&+ \left.\left(a^4\varepsilon_1^2 - 2a^2c^2\varepsilon_1\varepsilon_2 + c^4\varepsilon_2^2 \right)
			\vphantom{\frac{R^2}A}\right)	\\
		= A^2 \left( a^2\left(a^2 \frac 2A - 2c^2 \frac{R^2}A - 2 \frac 1y \right) \varepsilon_1\right.
		&+ c^2\left(-2a^2\frac 1A + 2c^2 \frac{R^2}A - 2 \frac 1y\right) \varepsilon_2	\\
		&+ \left.\left(a^2\varepsilon_1 - c^2\varepsilon_2 \right)^2
			\vphantom{\frac{R^2}A}\right).
	\end{align*}
	Putting everything together in one equation again we obtain
	\begin{align*}
		\lefteqn{\left( a^2 - 2\sigma ac + R^2c^2 - \frac Ay \right) \left( a^2 + 2\sigma ac + R^2c^2 - \frac Ay \right)}	\\
		&&= -A^2 \left( 2a^2\left(a^2 \frac 1A - c^2 \frac{R^2}A - \frac 1y \right) \varepsilon_1\right.
		+ 2c^2\left(-a^2\frac 1A + c^2 \frac{R^2}A - \frac 1y\right) \varepsilon_2	\\
		&&+ \left.\left(a^2\varepsilon_1 - c^2\varepsilon_2 \right)^2
			\vphantom{\frac R{\sin\varphi}}\right).
	\end{align*}
	We are now ready to prove \eqref{eq:matrixentries}. Choose $\nu \in \{\pm 1\}$ such that
	\begin{displaymath}
		\left\vert a^2 + 2\nu|\sigma| ac + R^2c^2 - \frac Ay \right\vert
		\le \left\vert a^2 - 2\nu|\sigma| ac + R^2c^2 - \frac Ay \right\vert.
	\end{displaymath}
	Then
	\begin{equation}\label{eq:bounding_ellipse}
		\begin{split}
			\left\vert a^2 + 2\nu|\sigma| ac + R^2c^2 - \frac Ay \right\vert^2
			\le& \left\vert a^2 - 2\sigma ac + R^2c^2 - \frac Ay \right\vert \left\vert a^2 + 2\sigma ac + R^2c^2 - \frac Ay \right\vert	\\
			\le& A^2 \max\left\{a^2,c^2\right\} \left( 2 \left( a^2\frac 1A + c^2\frac{R^2}A + \frac 1y\right) |\varepsilon_1|\right.	\\
			&+ 2 \left( a^2\frac 1A + c^2\frac{R^2}A + \frac 1y\right) |\varepsilon_2|	\\
			&\left.+ \max\{a^2,c^2\}\left(|\varepsilon_1| + |\varepsilon_2|\right)^2 \vphantom{\left(\frac 1A\right)}\right).
		\end{split}
	\end{equation}
	Note that $1/A \le 2/\sqrt{3}$ and $R^2/A \le 2R/\sqrt{3}$ as remarked on page \pageref{eq:abdependence}.
	We also have acquired a bound for $\max\{a^2,c^2\}$ in the beginning of the proof displayed in \eqref{eq:bound_a2_c2}.
	Therefore,
	\begin{align*}
		a^2\frac 1A + c^2 \frac{R^2}A + \frac 1y &\le \frac{4R+1}{\sqrt{3}}\frac 1y\frac 1A + \frac{4R+1}{\sqrt{3}}\frac 1y\frac{R^2}{A} + \frac 1y	\\
		&\le \frac{4R+1}{\sqrt{3}}\frac 1y\left(\frac 1A + \frac{R^2}{A} + \frac 2{\sqrt{3}}\right)	\\
		&\le \frac{4R+1}{\sqrt{3}} \frac{6R}{\sqrt{3}}\frac 1y.
	\end{align*}
	Using the bounds for $\varepsilon_1$ and $\varepsilon_2$ we get
	\begin{align*}
		2\left(a^2\frac 1A + c^2\frac{R^2}A + \frac 1y\right)|\varepsilon_1|
		&\le 2\frac{4R+1}{\sqrt{3}} \frac{6R}{\sqrt{3}} \frac 1y\frac{2}{\sqrt{3}A}\varepsilon
		\le 10\frac{4R+1}{\sqrt{3}} \frac \varepsilon{y}, \\
		2\left(a^2\frac 1A + c^2\frac{R^2}A + \frac 1y\right)|\varepsilon_2|
		&\le 2\frac{4R+1}{\sqrt{3}} \frac{6R}{\sqrt{3}} \frac 1y 4\varepsilon
		\le 28R\frac{4R+1}{\sqrt{3}} \frac \varepsilon{y}
	\end{align*}
	and
	\begin{displaymath}
		\left(|\varepsilon_1| + |\varepsilon_2|\right)^2
		\le \left(\frac 1A \frac 2{\sqrt 3}\varepsilon + 4\varepsilon\right)^2
		\le \varepsilon^2\left(\frac 1A \frac 2{\sqrt 3} + 4\right)^2
		\le \varepsilon^2\left(\frac 4{3} + 4\right)^2
		\le 29 \varepsilon^2
		\le 11 \varepsilon
	\end{displaymath}
	since $\varepsilon \le \sqrt{3}/5$.
	Using these inequalities for \eqref{eq:bounding_ellipse} and applying \eqref{eq:bound_a2_c2} again we obtain
	\begin{align*}
		\left\vert a^2 + 2\nu|\sigma| ac + R^2c^2 - \frac Ay \right\vert^2
		&\le A^2 \max\left\{a^2,c^2\right\} \left( 38R\frac{4R+1}{\sqrt{3}} \frac\varepsilon y
		+ 11\max\{a^2,c^2\}\varepsilon \right)	\\
		&\le 49 A^2 R \left(\frac{4R+1}{\sqrt{3}}\right)^2 \frac \varepsilon{y^2}.
	\end{align*}
	Taking the square--root on both sides gets us
	\begin{align*}
		\left\vert a^2 + 2\nu|\sigma| ac + R^2c^2 - \frac Ay \right\vert
		&\le 7 A R^{1/2} \left(\frac{4R+1}{\sqrt{3}}\right) \frac{\varepsilon^{1/2}}{y}.
	\end{align*}
	Using $A \le R$ and $y = \Im(\tau)$ we get
	\begin{align*}
		\left\vert a^2 + 2\nu|\sigma| ac + R^2c^2 - \frac Ay \right\vert
		&\le 7R^2 \left(\frac{4R+1}{\sqrt 3}\right) \frac{\varepsilon^{1/2}}{\Im(\tau)}.
	\end{align*}
	This proves \eqref{eq:matrixentries}.
\end{proof}

Note that the estimates might be improved slightly, especially when $\xi = \zeta$ or $\xi = \zeta^2$ with
$\zeta = e^{2\pi i /6}$.\\
We want to use the last lemma to prove the following proposition.
\begin{prop}\label{prop:clusterbound}
	Let $N$ be an integer, and let $E_0$ be an elliptic curve, and 
	$\xi \in \bar\F$. Further, assume that $0 \le \varepsilon \le (100^{-1}|\xi|^{-3}\Im(\xi))^2$.
	Then the number of $\tau \in \bar\F$
	with $|\xi - \tau| \le \varepsilon$ and such that $E_0$ is $N$--isogenous to a curve
	corresponding to $j(\tau)$ is bounded by
	\begin{displaymath}
		10^7|\tau_0||\xi|^5 \left( \sqrt{N} \sigma_0(N) + \sqrt\varepsilon \psi(N) \right).
	\end{displaymath}
\end{prop}
\noindent For the remainder of the section we are going to prove this proposition.

For fixed $\tau \in \H$ with bounded real part we want to bound the number of matrices that
satisfy the conditions in the lemma.
For this we define
\begin{displaymath}
	\M(\xi;x;y;\varepsilon) = \#\{ \gamma \in \SL_2(\Z) ; \exists\tau = \tilde x+iy, |\tilde x|\le |x|,
												|\gamma\tau - \xi| \le \varepsilon \text{ and } \gamma\tau \in \bar\F \}.
\end{displaymath}
Note that the last lemma tells us that all $\tau$ on horizontal lines in the upper half--plane satisfy the
same equation for $(a,c)$. Thus, if we look at horizontal line segments the number $\M(\xi;x;y;\varepsilon)$
can be bounded independent in terms of $x$.

If $\gamma$ is as in the last lemma, then the first column $(a,c)$ is close to one of the two ellipses
\begin{displaymath}
	X^2 \pm 2|\Re(\xi)|XY + |\xi|^2Y^2 = \frac{\Im(\xi)}{\Im(\tau)}.
\end{displaymath}
More precisely, we have
\begin{equation}\label{eq:ellipse}
	\left| \lambda - \frac{\Im(\xi)}{\Im(\tau)}\right| \le 50 |\xi|^3 \frac{\varepsilon^{1/2}}{\Im(\tau)},
		\quad \text{where} \quad \lambda = a^2 \pm 2|\Re(\xi)|ac + |\xi|^2c^2.
\end{equation}

We need an upper bound for the number $N(\Im(\tau),\varepsilon)$ of lattice points $(a,c) \in \Z^2$ that satisfy \eqref{eq:ellipse}.
Each of these points lies in a neighborhood of an ellipse defined above. We are going to use a result by Davenport \cite{davenport}.
The following theorem is a special case of the result of Davenport.

\begin{theorem}\label{thm:davenport}
	Let $\mathcal{R}$ be a region in the two--dimensional plane with smooth boundary. If $V(\mathcal{R})$ denotes
	the volume of $\mathcal{R}$ and $N(\mathcal{R})$ the number of points with integral coordinates in $\mathcal{R}$,
	then
	\begin{displaymath}
		|N(\mathcal{R}) - V(\mathcal{R})| < 4(L+1),
	\end{displaymath}
	where $L$ is the length of the boundary of $\mathcal{R}$.
\end{theorem}

Thus, we need to compute the volume and the circumference of the ellipses that bound the given neighborhood.
Let us assume that
\begin{displaymath}
	\varepsilon \le \left(\frac{\Im(\xi)}{100|\xi|^3}\right)^2
\end{displaymath}
is small enough.
We consider the case when $\nu = 1$. The ellipses are then given by
\begin{displaymath}
	E_\pm \colon A a^2 + B ab + C b^2 = 1
\end{displaymath}
with
\begin{align*}
	A = \frac{\Im(\tau)}{\Im(\xi)\pm 50|\xi|^3\varepsilon^{1/2}},
	\quad B = \frac{2|\Re(\xi)|\Im(\tau)}{\Im(\xi)\pm 50|\xi|^3\varepsilon^{1/2}},
	\quad C = \frac{|\xi|^2\Im(\tau)}{\Im(\xi)\pm 50|\xi|^3\varepsilon^{1/2}}.
\end{align*}
The area of the bigger ellipse is then given by
\begin{displaymath}
	\vol(E_+) = \frac{2\pi}{\sqrt{4AC - B^2}} 
	= \pi \frac{\Im(\xi) + 50|\xi|^3\varepsilon^{1/2}}{\Im(\tau)\sqrt{|\xi|^2 - \Re(\xi)^2}}
	= \pi \frac{\Im(\xi) + 50|\xi|^3\varepsilon^{1/2}}{\Im(\tau)\Im(\xi)}.
\end{displaymath}
Similarly, we have
\begin{displaymath}
	\vol(E_-) = \pi \frac{\Im(\xi) - 50|\xi|^3\varepsilon^{1/2}}{\Im(\tau)\Im(\xi)}
\end{displaymath}
for the smaller ellipse.

\noindent We now want to bound the circumference of $E_\pm$. For this we will use the
following lemma.
\begin{lem}
	Let $E$ be an ellipse given by $Aa^2 + Bac + Cc^2 = 1$. Then the circumference $L$
	of $E$ is bounded by
	\begin{displaymath}
		L \le \sqrt{2(A+C)}\vol(E).
	\end{displaymath}
\end{lem}

\begin{proof}
	To prove this we rotate the ellipse, so that the new equation becomes
	\begin{equation}\label{eq:ellipse_defeq}
		A'a^2 + C'b^2 = 1.
	\end{equation}
	The coef{}f{}icients are given by
	\begin{displaymath}
		A' = \frac{A+C}2 + \frac{A-C}2\cos(2\theta) - \frac B2\sin(2\theta)
	\end{displaymath}
	and
	\begin{displaymath}
		C' = \frac{A+C}2 - \frac{A-C}2\cos(2\theta) + \frac B2\sin(2\theta),
	\end{displaymath}
	where $\theta$ satisfies $\cot 2\theta = \frac{A-C}B$ or $\tan 2\theta = \frac{B}{A-C}$.
	Note if $B = 0$, we have $\theta = 0$, so that $A'=A$ and $C'=C$.
	Now the circumference of an ellipse in the form of \eqref{eq:ellipse_defeq}
	can be estimated by
	\begin{displaymath}
		L \le \sqrt{2} \pi \sqrt{\frac 1{A'} + \frac 1{C'}}
		\le \sqrt{2} \pi \sqrt{\frac{A'+C'}{A'C'}}
		= 2\sqrt{2} \pi \sqrt{\frac{A'+C'}{4A'C'}}.
	\end{displaymath}
	But if we put $B'=0$, then $A'+C' = A+C$ and $4A'C' = 4A'C' - B'^2 = 4AC - B^2$
	since the discriminant is an invariant.
	Thus
	\begin{displaymath}
		L \le \sqrt{2} \sqrt{A+C}\frac{2\pi}{4AC-B^2}
		= \sqrt{2} \sqrt{A+C}\vol(E),
	\end{displaymath}
	as desired.
\end{proof}

If $L_+$ denotes the circumference of $E_+$, then we have by the previous lemma
\begin{align*}
	L_+ &\le \sqrt{2}\pi \sqrt{\frac{\Im(\tau) + |\xi|^2\Im(\tau)}{\Im(\xi) + 50|\xi|^3\varepsilon^{1/2}}}
	\frac{\Im(\xi) + 50|\xi|^3\varepsilon^{1/2}}{\Im(\tau)\Im(\xi)}	\\
	&\le \sqrt{2}\pi \frac{\sqrt{1 + |\xi|^2}}{\Im(\xi)} \sqrt{\frac{\Im(\xi) + 50|\xi|^3\varepsilon^{1/2}}{\Im(\tau)}}.
\end{align*}
Now we use the bound on $\varepsilon$ to get
\begin{align*}
	L_+ &\le \sqrt{2}\pi \frac{\sqrt{1 + |\xi|^2}}{\Im(\xi)} \sqrt{\frac{\Im(\xi) + \frac 12 \Im(\xi)}{\Im(\tau)}}	\\
	&= \sqrt{3}\pi \sqrt{\frac{1 + |\xi|^2}{\Im(\xi)}} \frac{1}{\sqrt{\Im(\tau)}}.
\end{align*}
We have $\frac{|\xi|}{\Im(\xi)} \le \frac 2{\sqrt{3}}$ since $\xi$ is in the fundamental domain.
Hence $\frac{|\xi|^2}{\Im(\xi)} \le \frac 4{3}\Im(\xi)$ and therefore
\begin{align*}
	\frac{1 + |\xi|^2}{\Im(\xi)} = \frac{1}{\Im(\xi)} + \frac{|\xi|^2}{\Im(\xi)} \le 
	\frac 2{\sqrt{3}} + \frac 4{{3}}\Im(\xi) \le \frac 83\Im(\xi).
\end{align*}
Using this for the bound of $L_+$ yields
\begin{displaymath}
	L_+ \le 2\pi \frac{\sqrt{2\Im(\xi)}}{\sqrt{\Im(\tau)}}.
\end{displaymath}

Clearly this bound also holds for $L_-$, the circumference of the smaller ellipse.

Let $N(E_\pm)$ denote the number of lattice points contained in $E_\pm$ as defined in Theorem \ref{thm:davenport}.
By this same theorem, the number of points contained
in the elliptical annulus can be estimated by
\begin{align*}
	N(E_+) - N(E_-) &= N(E_+) - \vol(E_+) - (N(E_-) - \vol(E_-)) + \vol(E_+) - \vol(E_-)	\\
	&\le 4(L_++1) + 4(L_-+1) + \vol(E_+) - \vol(E_-)	\\
	&\le 8(L_+ + 1) + \frac{100\pi|\xi|^3\varepsilon^{1/2}}{\Im(\tau)\Im(\xi)}	\\
	&\le 16\pi \frac{\sqrt{2\Im(\xi)} + \sqrt{\Im(\tau)}}{\sqrt{\Im(\tau)}}
			+ \frac{100\pi|\xi|^3\varepsilon^{1/2}}{\Im(\tau)\Im(\xi)}.
\end{align*}
Therefore, a bound for $N(\Im(\tau),\varepsilon)$ is given by twice this number since the ellipse for $\nu = -1$
gives the same bound.

To obtain a bound for the number of matrices satisfying the conditions in Lemma \ref{lem:ellipse},
we need to estimate the possible pairs
$(b,d)$ when $(a,c)$ is fixed.
Let $(a,c)$ be fixed, and assume that $(b,d)$ and $(b',d')$
satisfy $ad-bc=1$ and $ad'-b'c=1$, respectively. Then $(b-b',d-d') = (ak,ck)$ for some integer $k$.
Lemma \ref{lem:ellipse} now implies
\begin{displaymath}
	|k| \le 2|\Re(\tau)| + 2\frac{4|\xi|+1}{\sqrt 3} \le 2|\Re(\tau)| + 6|\xi|.
\end{displaymath}
Thus, $\M(\xi;x;y;\varepsilon)$ is bounded by
\begin{align}\label{eq:numberofmatrices}
	N(y,\varepsilon) &\cdot \left( 2\cdot(2x + 6|\xi|) + 1\right)
	\le N(y,\varepsilon) \cdot (4x + 13|\xi|) \notag\\
	&\le 2\left(16\pi \frac{\sqrt{2\Im(\xi)} + \sqrt{y}}{\sqrt{y}}
			+ \frac{100\pi|\xi|^3\varepsilon^{1/2}}{y\Im(\xi)}\right)
				\left(4x + 13|\xi|\right).
\end{align}

We now want to apply this result to estimate the number of points close to a fixed point which are
given by a cyclic isogeny of degree $N$.
Let $\tau_0 \in \H$ be fixed. Let $N \in \N$.
We will be working with matrices $M$ of the form $\begin{smatrix}m&l\\0&n\end{smatrix}$
with $N = mn$ and $0 \le l < n$. We will denote $M.\tau_0$ by $\tau_M$.
We want to bound the number of points $\tau_M$ satisfying $|\tilde\tau_M - \xi| \le \varepsilon$ with $\tilde\tau_M$ in
the $\SL_2(\Z)$--orbit of $\tau_M$ and in $\bar\F$.
For this we momentarily fix a divisor $n$ of $N$ with $n \ge \sqrt{N}$ and a matrix $M$
with $M = \begin{smatrix}m&l\\0&n\end{smatrix}$ and $0 \le l < n$. Then $y := \Im(\tau_M) = \frac mn\Im(\tau_0)$
for any $0 \le l < n$. Figure \ref{fig:isogenouspts} shows an example with $\tau = 1+i$ and $N=10$.
	
\begin{figure}\label{fig:isogenouspts}
	\begin{center}
		\pgfmathsetmacro{\myxlow}{-1}
		\pgfmathsetmacro{\myxhigh}{3}
		\pgfmathsetmacro{\myiterations}{2}
		\begin{tikzpicture}[scale=1.5]
			\draw[-latex](\myxlow-0.1,0) -- (\myxhigh+0.2,0);
			\pgfmathsetmacro{\succofmyxlow}{\myxlow+0.5}
			\foreach \x in {\myxlow,\succofmyxlow,...,\myxhigh}
			{   \draw (\x,0) -- (\x,-0.05) node[below,font=\tiny] {\x};
			}
			\foreach \y  in {0.5,1,...,2}
			{   \draw (0.02,\y) -- (-0.02,\y) node[left,font=\tiny] {\pgfmathprintnumber{\y}};
			}
			\draw[-latex](0,-0.05) -- (0,2.6);

			\coordinate (zeta) at (1/2,{sqrt(3)/2});
			\coordinate (zeta2) at (-1/2,{sqrt(3)/2});

			\draw[domain=60:120] plot ({cos(\x)}, {sin(\x)});
			\path[bottom color=black,top color=white] (zeta) rectangle +(.2pt,2.5cm);
			\path[bottom color=black,top color=white] (zeta2) rectangle +(-.2pt,2.5cm);

			\coordinate (tau) at (1, 1);

			\draw (tau) circle (.3pt) node[below,font=\tiny] {$\tau$};

			\foreach \n in {2,5,10}
			{
			\pgfmathsetmacro\ratio{10/\n/\n}
			\pgfmathsetmacro\up{\n-1}
			\foreach \l in {0,...,\up}
			{
				\pgfmathsetmacro\ratiol{\l/\n}
				\filldraw ($\ratio*(tau) + (\ratiol,0)$) circle (.3pt) {};
			}
			}

		\end{tikzpicture}
	\end{center}
	\caption{\small$\tau_0$ and all except one $\tau_M$ for $N=10$}
\end{figure}
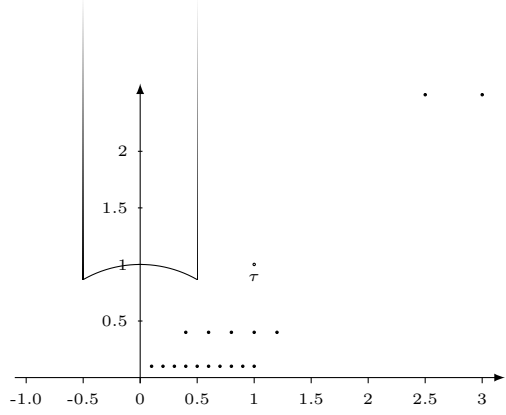

Since $y$ does not depend on $l$ and $|\Re(\tau_M)| \le |\Re(\tau_0)|+1$, the bound
for $\M(\xi;|\Re(\tau_0)|+1;y;\varepsilon)$ is independent of $l$. This number
does not estimate all the $\tau_M$ that are translated close to $\xi$ as
we will see later.
The bound in \eqref{eq:numberofmatrices} translates to
\begin{align*}
	\M(\xi;|\Re(\tau_0)|+1;y;\varepsilon)
	\le 8\left(4\pi \frac{\sqrt{2\Im(\xi)} + \sqrt{\frac mn\Im(\tau_0)}}{\sqrt{\Im(\tau_0)}}\sqrt{\frac nm}
			+ \frac{25\pi|\xi|^3\varepsilon^{1/2}}{\Im(\tau_0)\Im(\xi)}\frac nm\right)	\\
				\cdot\left(4|\Re(\tau_0)| + 13|\xi| + 4\right).
\end{align*}
But $\frac mn \le 1$ since $n \ge \sqrt{N}$ and hence
\begin{align*}
	\M(\xi;&|\Re(\tau_0)|+1;y;\varepsilon)	\\
	&\le 8\left(4|\Re(\tau_0)| + 17|\xi|\right)
		\left(4\pi \frac{\sqrt{2\Im(\xi)} + \sqrt{\Im(\tau_0)}}{\sqrt{\Im(\tau_0)}}\sqrt{\frac nm}
			 +\frac{25\pi|\xi|^3\varepsilon^{1/2}}{\Im(\tau_0)\Im(\xi)}\frac nm\right)
					\\
		&\le 8\left(4|\Re(\tau_0)| + 17|\xi|\right)
		\left(8\pi \frac{\sqrt{2\Im(\xi)} + \sqrt{\Im(\tau_0)}}{\sqrt{\Im(\tau_0)}}\sqrt{\frac nm}
			 +\frac{35\pi|\xi|^2\varepsilon^{1/2}}{\Im(\tau_0)}\frac nm\right)
\end{align*}
if we also apply $|\xi|^2/\Im(\xi) \le 4|\xi|/3$.
Further we get
\begin{align}\label{eq:taul_matrices}
	\M(\xi;|\Re(\tau_0)|+1;y;\varepsilon)
		\notag
		\le 64\pi&\left(4|\Re(\tau_0)| + 17|\xi|\right)	\\
		&\cdot \max\left\{\frac{\sqrt{2\Im(\xi)} + \sqrt{\Im(\tau_0)}}{\sqrt{\Im(\tau_0)}},5\frac{|\xi|^2}{\Im(\tau_0)}\right\}	\\
				&\cdot\left(\sqrt{\frac nm} + \frac nm \varepsilon^{1/2}\right).
				\notag
\end{align}

Now different $\tau_M$ (entry $l$ different) can be translated close to $\xi$ by the same matrix,
so we have to restrict those.
So if $\tau_M$ is translated into the disc around $\xi$ by a matrix $\begin{smatrix}a&b\\c&d\end{smatrix}$ then
the real part $x$ of $\tau_M$ satisfies
\begin{displaymath}
	\left|cx+d\right||c| \le 3|\xi|
	\quad \text{and} \quad
	\left|ax+b\right||a| \le 3|\xi|
\end{displaymath}
by \eqref{eq:dequation} and \eqref{eq:bequation}.
Assume that $c\not=0$. Then $\left\vert x+d/c\right\vert \le 3|\xi| c^{-2}$, so that
$x$ lies in an interval $I$ with center $-d/c$ and of length bounded by $6|\xi| c^{-2}$.
This implies
\begin{displaymath}
	\left\vert\left\{l \in \{0, \dotsc, n-1 \} : (m\Re(\tau_0) + l)/n \in I \right\}\right\vert
	\le n|I| + 1 \le 6|\xi| \frac{n}{c^2} + 1.
\end{displaymath}
A similar result is obtained if $a\not=0$. So in any case
\begin{equation}\label{eq:bound_taul_ac}
	\left\vert\left\{l \in \{0, \dotsc, n-1 \} : (m\Re(\tau_0) + l)/n \in I \right\}\right\vert
	\le 6|\xi| \frac{n}{\max\{|a|,|c|\}^2} + 1
\end{equation}
independent of whether the interval is centered around $-d/c$ or $-b/a$.
Moreover, $\max\{|a|,|c|\}^2$ can be bounded by
\begin{align*}
	3|\xi|^2\max\{|a|,|c|\}^2 &\ge a^2 + \nu 2 |\Re(\xi)| ac + |\xi|^2 c^2	\\
	&\ge \frac{\Im(\xi) - 50|\xi|^3\varepsilon^{1/2}}{y}
	= \frac{\Im(\xi) - 50|\xi|^3\varepsilon^{1/2}}{\Im(\tau_0)}\frac nm	\;,
\end{align*}
where the last inequality follows from Equation \eqref{eq:matrixentries}.
Using the upper bound on $\varepsilon$ we obtain
\begin{displaymath}
	3|\xi|^2\max\{|a|,|c|\}^2 \ge \frac{\Im(\xi)}{2\Im(\tau_0)}\frac nm \;,
\end{displaymath}
and hence
\begin{equation}\label{eq:bound_taul}
	6|\xi| \frac{n}{\max\{|a|,|c|\}^2}
	\le 36|\xi|^3 \frac{\Im(\tau_0)}{\Im(\xi)} m
	\le 50|\xi|^2 \Im(\tau_0) m
\end{equation}
since $|\xi|/\Im(\xi) \le 2/\sqrt{3}$.

Recall that the matrix $M$ is of the form $\begin{smatrix}m&l\\0&n\end{smatrix}$
with $N = mn$ and $0 \le l < n$  and $\tau_M = M\tau_0$.
As before, $\tilde\tau_M$ is in the $\SL_2(\Z)$--orbit of $\tau_M$ and in $\bar\F$.

Let $\Lambda(\tau_0;N;\varepsilon)$ be the set of $\tau_M$ satisfying $|\tilde\tau_M - \xi| \le \varepsilon$,
where $\tau_M$ is as before.
The number of elements in $\Lambda(\tau_0;N;\varepsilon)$ is surely bounded by the number of matrices $M$ with
lower right entry greater than $\sqrt N$ satisfying the condition
plus the total number of matrices with $n \le \sqrt{N}$. The latter is bounded by
\begin{equation*}
	\sum_{\twostack{n|N}{0 < n \le \sqrt{N}}} n \le \sqrt{N} \sum_{n|N} 1 = \sqrt{N} \sigma_0(N).
\end{equation*}
For $n \le \sqrt{N}$ we are going to count $\tilde\tau_M$ independent of whether $|\tilde\tau_M - \xi| \le \varepsilon$
or not since the number $\sqrt{N}\sigma_0(N)$ does not grow too fast.
Now by the arguments we just made, we can bound the number of $\tau_M$ and thus the total number
of points in $\Lambda(\tau_M;N;\varepsilon)$ as follows. Recall that $N=mn$.
\begin{align*}
	|\Lambda(\tau_0;N;\varepsilon)| \le& \sum_{\twostack{n|N}{n \ge \sqrt{N}}}
						\M\left(\xi;|\Re(\tau_0)|+1;\frac N{n^2}\Im(\tau_0);\varepsilon\right) \left(
	6|\xi| \frac{n}{\max\{|a|,|c|\}^2} + 1 \right)	\\
	&+ \sqrt{N}\sigma_0(N).
\end{align*}
Here, for fixed $n$ the number $\M\left(\xi;|\Re(\tau_0)|+1;\frac N{n^2}\Im(\tau_0);\varepsilon\right)$
bounds the matrices
that translate any $\tau_M$ of the form $\begin{smatrix}m&l\\0&n\end{smatrix}$ with varying $l$ close to $\xi$.
But since different $\tau_M$ can be translated into
the disc around $\xi$ by the same matrix we have to compensate this with the inequality in \eqref{eq:bound_taul_ac}.
This in turn can be estimated as displayed in \eqref{eq:bound_taul} so that
\begin{align*}
	|\Lambda(\tau_0;N;\varepsilon)|
	\le& \sum_{\twostack{n|N}{n \ge \sqrt{N}}}
						\M\left(\xi;|\Re(\tau_0)|+1;\frac N{n^2}\Im(\tau_0);\varepsilon\right)
	\left( 50|\xi|^2 \Im(\tau_0) \frac Nn + 1 \right) 	\\
	&+ \sqrt{N}\sigma_0(N).
\end{align*}
By the inequality for $\M\left(\xi;|\Re(\tau_0)|+1;\frac N{n^2}\Im(\tau_0);\varepsilon\right)$
in \eqref{eq:taul_matrices} and $1 \le m$ we get
\begin{align*}
	|\Lambda(\tau_0;N;\varepsilon)|
	\le \sqrt{N}\sigma_0(N) +
	\sum_{\twostack{m|N}{m \le \sqrt{N}}}m 
		\cdot \mathcal{I}(\tau_0,\xi)
			\cdot\left(\sqrt{\frac N{m^2}} + \frac N{m^2} \varepsilon^{1/2}\right)
\end{align*}
where
\begin{align*}
	\mathcal{I}(\tau_0,\xi) = 64\pi&\left(4|\Re(\tau_0)| + 17|\xi|\right)\left( 50|\xi|^2 \Im(\tau_0) + 1 \right) 	\\
	&\cdot\max\left\{\frac{\sqrt{2\Im(\xi)} + \sqrt{\Im(\tau_0)}}{\sqrt{\Im(\tau_0)}},5\frac{|\xi|^2}{\Im(\tau_0)}\right\}.
\end{align*}
We can continue the estimate
\begin{align*}
	|\Lambda(\tau_0;N;\varepsilon)|
	&\le \sqrt{N}\sigma_0(N) + \mathcal{I}(\tau_0,\xi)
	\sum_{\twostack{m|N}{m \le \sqrt{N}}} m
		\left(\frac{\sqrt{N}}{m} + \frac N{m^2} \varepsilon^{1/2}\right)	\\
	&= \sqrt{N}\sigma_0(N) + \mathcal{I}(\tau_0,\xi)
	\sum_{\twostack{m|N}{m \le \sqrt{N}}}
		\left(\sqrt{N} + \frac N{m} \varepsilon^{1/2}\right).
\end{align*}
We split the sum to get
\begin{align}
	\notag
	|\Lambda(\tau_0;N;\varepsilon)|
	&\le \sqrt{N}\sigma_0(N) + \mathcal{I}(\tau_0,\xi)
	\left(\sum_{\twostack{m|N}{m \le \sqrt{N}}}\sqrt{N} +
	\sum_{\twostack{m|N}{m \le \sqrt{N}}} \frac N{m} \varepsilon^{1/2}\right)	\\
	\notag
	&= \sqrt{N}\sigma_0(N)(1 + \mathcal{I}(\tau_0,\xi)) + \mathcal{I}(\tau_0,\xi)
	\varepsilon^{1/2} \sum_{\twostack{n|N}{n \ge \sqrt{N}}} n	\\
	&= \sqrt{N}\sigma_0(N)(1 + \mathcal{I}(\tau_0,\xi)) + \mathcal{I}(\tau_0,\xi)
	\varepsilon^{1/2} \sigma_1(N).
	\label{eq:isogeny_cluster_bound}
\end{align}

\begin{lem}
	Let $N$ be a positive integer. Then $\sigma_1(N) \le \frac{\pi^2}6 \psi(N)$.
\end{lem}

\begin{proof}
	It is well--known that $\sigma_1$ is multiplicative. The function $\psi$ is also multiplicative,
	see page 53 of \cite{LangEllFunc}.
	We have $\sigma_1(p^k) = \frac{p^{k+1}-1}{p-1}$ and $\psi(p^k) = p^{k-1}(p+1)$.
	Thus
	\begin{align*}
		\frac{\sigma_1(p^k)}{\psi(p^k)} = \frac{p^{k+1} - 1}{(p-1)(p+1)p^{k-1}} \le \frac{p^2}{p^2-1} = \frac 1{1-p^{-2}}.
	\end{align*}
	For general $N$ this yields
	\begin{align*}
		\frac{\sigma_1(N)}{\psi(N)} \le \prod_{p^k \Vert N} \frac 1{1-p^{-2}} \le \prod_{p} \frac 1{1-p^{-2}} = \zeta(2),
	\end{align*}
	where $\zeta$ denotes the Riemann zeta function. This proves the claim.
\end{proof}

\noindent Altogether we get
\begin{lem}
	Fix $\tau_0 \in \H$ and $\xi \in \bar\F$. Let $0 < \varepsilon \le \left(\frac{\Im(\xi)}{100|\xi|^3}\right)^2$.
	Let $\Lambda(\tau_0;N;\varepsilon)$ be the number of $M\tau_0$,
	$M = \begin{smatrix}m&l\\0&n\end{smatrix}$ with $N=mn$ and $0\le l < n$,
	that satisfy $|\gamma M\tau_0 - \xi| \le \varepsilon$ for some $\gamma \in \sl2z$.
	Then
	\begin{displaymath}
		|\Lambda(\tau_0;N;\varepsilon)|
		\le \sqrt{N}\sigma_0(N)(1 + \mathcal{I}(\tau_0,\xi)) + \frac{\pi^2}6 \mathcal{I}(\tau_0,\xi)
		\varepsilon^{1/2} \psi(N).
	\end{displaymath}
	with
	\begin{align*}
		\mathcal{I}(\tau_0,\xi) = 64\pi&\left(4|\Re(\tau_0)| + 17|\xi|\right)\left( 50|\xi|^2 \Im(\tau_0) + 1 \right) 	\\
		&\cdot\max\left\{\frac{\sqrt{2\Im(\xi)} + \sqrt{\Im(\tau_0)}}{\sqrt{\Im(\tau_0)}},5\frac{|\xi|^2}{\Im(\tau_0)}\right\}.
	\end{align*}
\end{lem}

To complete the proof of Proposition \ref{prop:clusterbound} we restrict $\tau_0$ to the fundamental domain.
Then $\Im(\tau_0) \ge \sqrt{3}/2$ and $|\Re(\tau_0)| \le 1/2$.
Therefore we get
\begin{align*}
	\frac{\sqrt{2\Im(\xi)} + \sqrt{\Im(\tau_0)}}{\sqrt{\Im(\tau_0)}}
	=1 + \frac{\sqrt{2\Im(\xi)}}{\sqrt{\Im(\tau_0)}}
	\le 1 + 2\sqrt{\Im(\xi)}
	\le 6|\xi|^2
\end{align*}
and hence
\begin{align*}
	\max\left\{1+\mathcal{I}(\tau_0,\xi),\frac{\pi^2}6\mathcal{I}(\tau_0,\xi)\right\}
	\le 64\pi&\left(2 + 17|\xi|\right) \cdot 60|\xi|^2 |\tau_0|
	\cdot 6|\xi|^2
	\le 10^7|\tau_0| |\xi|^5.
\end{align*}
Recall that an $N$--isogeny (i.e.~a cyclic isogeny of degree $N$) is related to a matrix
of the form $M = \begin{smatrix}m&l\\0&n\end{smatrix}$ with $N=mn$, $0\le l < n$ and
$\gcd(m,n,l) = 1$. We have considered such matrices without a condition
on the greatest common divisor.
Therefore we are done with the proof of Proposition \ref{prop:clusterbound}.

\section{Bounding the height}

Recall that we have fixed an elliptic curve without complex multiplication defined over a number field $K$
and $j_0$ is its $j$--invariant.
Two points in the fundamental domain are in the same Hecke orbit if there exists an isgoeny between them.
We are going to compare the Galois orbit of $j_0$ to the Hecke orbit of all conjugates of $E_0$.
We now want to bound the number of elements in $\Gamma(\xi,\varepsilon)$. For this we use the connection
between the isogeny orbit and the Galois orbit of Serre's open image theorem. See Th\'eor\`eme 3 in §4 of
\cite{openimage}.

More precisely, we will be using a version proved by Lombardo \cite{lombardo}, that gives us an explicit bound.
Let $G_K = \Gal(\bar{K}/K)$ be the absolute Galois group of $K$. Recall that $G_K$ acts on the
$N$--torsion points of $N$, and we thus get a representation
\begin{displaymath}
	\rho_N\colon G_K \rightarrow \Aut(E[N]).
\end{displaymath}
The group $\Aut(E[N])$ is isomorphic to $\GL_2(\Z/N\Z)$.
It is possible to choose a suitable basis of $\GL_2(\Z/N\Z)$ so that
we obtain a representation
\begin{displaymath}
	\rho_\infty\colon G_K \rightarrow \GL_2(\hat\Z)
\end{displaymath}
after taking the inverse limit (over $N$.)
Serre proved in \cite{openimage} that $[ \GL_2(\hat\Z) : \rho_\infty(G_K)]$ is finite.
The result by Lombardo implies
\begin{equation}\label{eq:lombardo}
	\left[ \GL_2(\hat\Z) : \rho_\infty(G_K) \right]
		< \gamma_1 \cdot [K:\Q]^{\gamma_2} \cdot \max\left\{1,h(E),\log[K:\Q]\right\}^{2\gamma_2}
\end{equation}
where $\gamma_1 = \exp\left(10^{21483}\right)$ and $\gamma_2 = 2.4 \cdot 10^{10}$.
In particular, we obtain
\begin{displaymath}
	\left[ \GL_2(\Z/N\Z) : \rho_N(G_K) \right]
		< \gamma_1 \cdot [K:\Q]^{\gamma_2} \cdot \max\left\{1,h(E),\log[K:\Q]\right\}^{2\gamma_2}.
\end{displaymath}

Note that Lombardo's result actually uses the original definition of the Faltings height.
This information was acquired through a private conversation with the author.
Since the original definition of the Faltings height
is smaller than the normalization of Deligne,
we can just substitute $h(E)$ into Lombardo's result.

The cyclic isogenies of degree $N$ correspond in a one--to--one fashion to the cyclic subgroups of order $N$
in $\Z/N\Z \times \Z/N\Z \simeq E[N]$.
The action of $\GL_2(\Z/N\Z)$ on these subgroups is transitive as the next lemma states.
We start with some group theory. We denote by $\varphi$ Euler's totient function
given by $\varphi(N) = \#(\Z/N\Z)^\times = N \prod_{p|N} (1-1/p)$, where the product
runs over all primes $p$ dividing $N$.
Recall $\psi(N) = N \prod_{p|N} (1+ 1/p)$.

\begin{lem}
	The cardinality of $\GL_2(\Z/N\Z)$ is equal to $\varphi(N)^2 \psi(N) N$.
	Let $\Delta \subset \GL_2(\Z/N\Z)$ denote the subgroup of upper triangular matrices.
	Then $\#\Delta = N\varphi(N)^2$.
	There are $\psi(N)$ cyclic subgroups of $\Z/N\Z \times \Z/N\Z$.
	The group $\GL_2(\Z/N\Z)$ acts transitively on the cyclic subgroups of order $N$ in
	$(\Z/N\Z)^2$.
\end{lem}

\begin{lem}\label{lem:serre_index_bound}
	Let $E/K$ be an elliptic curve, $N$ an integer, and $\Phi \subset E[N]$ a cyclic subgroup of $N$--torsion points.
	Put $B = |\{ \sigma(\Phi) : \sigma \in \Gal(\bar K/K)\}|$. Then we have
	\begin{displaymath}
		\frac{\psi(N)}B \le \left[ \GL_2(\Z/N\Z) : \rho_N(G_K) \right].
	\end{displaymath}
\end{lem}

\begin{proof}
	Suppose $\Phi$ is generated by $P \in E[N]$. After choosing a basis, we may assume that $P$ corresponds
	to $(1,0)$ in $\Z/N\Z \times \Z/N\Z$. For any $\sigma \in \Gal(\bar K/K)$, the group $\sigma(\Phi)$
	is generated by a point $(a,c) \in Z/N\Z \times \Z/N\Z$
	where $\left(\begin{smallmatrix}a&b\\c&d\end{smallmatrix}\right)$ is the image of $\sigma$ under $\rho_N$.	\\
	Let $\Delta$ be the subgroup of upper triangular matrices of $\GL_2(\Z/N\Z)$. The equality $\sigma(\Phi) = \Phi$
	holds if and only if $\sigma$ is mapped into $\Delta$ under $\rho_N$. We thus have
	\begin{displaymath}
		B = \frac{\#\im \rho_N}{\#(\Delta \cap \im \rho_N)} \ge \frac{\#\im \rho_N}{\#\Delta} = \frac{\#\im \rho_N}{N\varphi(N)^2}.
	\end{displaymath}
	This implies
	\begin{align*}
		\frac{\psi(N)}B \le \frac{\psi(N)\varphi(N)^2 N}{\#\im \rho_N}
		&= \frac{\psi(N)\varphi(N)^2 N}{\#\GL_2(\Z/N\Z)} \left[ \GL_2(\Z/N\Z) : \rho_N(G_K) \right]	\\
		&=\left[ \GL_2(\Z/N\Z) : \rho_N(G_K) \right],
	\end{align*}
	as desired.
\end{proof}

We want to estimate a Mertens' type of sum. In fact, we are going to use a result by Mertens.

\begin{lem}
	Let $n \ge 4$ be a positive integer. Then
	\begin{displaymath}
		\sum_{p | n} \frac{\log p}p \le 5.25 \log\log n, 
	\end{displaymath}
	where the sum runs over all prime divisors of $n$.
\end{lem}

\begin{proof}
	The function $\log x/x$ is decreasing on $(e,\infty)$.
	Note that $(\log 2)/2 < (\log 3)/3$. So let $n = p^a$ be a prime power with $p \not= 2$. Then
	\begin{displaymath}
		\frac{\log p}p \le \frac{\log 3}3 \le 5\log\log 3 \le 5\log\log p
	\end{displaymath}
	and the claim holds. If $n = 2^a$ with $a \ge 2$, then
	\begin{displaymath}
		\frac{\log 2}2 < 1 \le 5\log\log(4).
	\end{displaymath}
	Now let $n = p^a q^b$ with different primes $p,q$ and $a,b \ge 1$. We have $(\log 5)/5 < (\log 2)/2$
	and $(\log p)/p < 0.5$. Thus
	\begin{displaymath}
		\frac{\log p}p + \frac{\log q}q \le \frac{\log 2}2 + \frac{\log 3}3 \le 1 < 5\log\log 6 \le 5\log\log n.
	\end{displaymath}
	So the claim is true for all $4\le n \le 29$.
	let us now assume that $n$ is composite with $\omega(n) \ge 3$.
	We can bound the sum by looking at the first $\omega(n)$ primes
	\begin{displaymath}
		\sum_{p | n} \frac{\log p}p
		\le \frac{\log p_1}{p_1} + \frac{\log p_2}{p_2} + \cdots + \frac{\log p_{\omega(n)}}{p_{\omega(n)}}.
	\end{displaymath}
	Note that $(\log 2)/2 < (\log 3)/3$, so that if $3$ occurs in the prime decomposition of $n$ and $2$
	does not, we can just estimate the largest prime divisor of $n$ by $(\log 2)/2$ and get the same inequality.
	It is a well--known result by Cipolla in \cite{cipolla1902determinazione}, that the $n$--th prime $p_n$ is
	bounded from above by $n(\log n + \log \log n)$ for sufficiently large $n$.
	Indeed Rosser proved in Theorem 2 of \cite{rosser1939n} that $p_n \le n(\log n + 2\log\log n)$ for all $n \ge 4$.
	Also compare to the bound in \cite{rosser1962approximate}.
	Hence $p_n \le 2n\log n$ for all $n \ge 3$ since this bound also holds for $p_3 = 5$.
	Since we have $\omega(n) \ge 3$ we can apply this to the last inequality to obtain
	\begin{displaymath}
		\sum_{p | n} \frac{\log p}p \le \sum_{p \le 2\omega(n)\log\omega(n)} \frac{\log p}{p}.
	\end{displaymath}
	By Mertens' Theorem (see \cite{mertens1874beitrag}) the sum on the right--hand side is bounded by
	\begin{displaymath}
		\sum_{p \le 2\omega(n)\log\omega(n)} \frac{\log p}{p} \le 2\log(2\omega(n)\log\omega(n))
	\end{displaymath}
	for all $n \ge 1$ composite of at least $3$ distinct primes.
	We have the trivial inequality
	\begin{displaymath}
		\omega(n) \le \frac{\log n}{\log 2}.
	\end{displaymath}
	This gives us
	\begin{align*}
		\sum_{p | n} \frac{\log p}p
		&\le 2\log\left(2\frac{\log n}{\log 2}\log\frac{\log n}{\log 2}\right)	\\
		&\le 2\log\log n + 2\log\left(\log \log n - \log \log 2\right) + 2.12
	\end{align*}
	and if $n \ge 5$ this gets us
	\begin{align*}
		\sum_{p | n} \frac{\log p}p
		&\le 2\log\log n + 2\log\left(2\log \log n\right) + 2.12	\\
		&\le 2\log\log n + 2\log\log\log n + 3.51.
	\end{align*}
	But we have $\log\log n \le \frac{36}{100} \log n$ since $x \mapsto (\log \log x)/\log x$ is decreasing
	for $x \ge 16$ and $(\log \log 30)/(\log 30) < 0.36$.
	Because  of $3.51 + 2\log 0.36 < 1.25 \log \log 30$ we obtain
	\begin{displaymath}
		\sum_{p \le 2\omega(n)\log\omega(n)} \frac{\log p}{p} \le 5.25\log\log n,
	\end{displaymath}
	as desired.
\end{proof}

\begin{prop}
	Let $E/\bar{\Q}$ and $E_0/K$ be elliptic curves without CM such that there exists a cyclic isogeny
	of degree $N$ from $E_0$ to $E$.
	Let $\rho_N$ be the Galois representation associated to $E_0$.
	If $N \ge 4$, we have
	\begin{displaymath}
		h(E) \ge h(E_0) + \frac 12 \log N - 7 \cdot [\GL_2(\Z/N\Z):\rho_N(G_K)]\log\log N.
	\end{displaymath}
\end{prop}

\begin{proof}
	We denote by $h(E)$ and $h(E_0)$ the stable Faltings height of $E$ and $E_0$, respectively.
	(The stated inequality does not depend on the normalization of the Faltings height.)
	We consider the action of $\Gal(\bar\Q/K)$ on the set of $\bar\Q$--isomorphism classes of
	elliptic curves.
	Let $E = E_1, \ldots, E_{\psi(N)}$ be representatives of elliptic curves that are $N$--isogenous
	to $E_0$.
	Note that the group $\Gal(\bar\Q/K)$ acts on the set $\{ E_1, \ldots, E_{\psi(N)} \}$.
	By Corollaire 3.3 in \cite{autissier} we have
	\begin{displaymath}
		\frac 1{\psi(N)}\sum_{i=1}^{\psi(N)} h(E_i) = h(E_0) + \frac 12 \log N - \lambda_N
	\end{displaymath}
	where $N = p_1^{\alpha_1} \cdots p_r^{\alpha_r}$ and
	$\lambda_N = \sum_{i=1}^r \frac{p_i^{\alpha_i} - 1}{(p_i^2-1)p_i^{\alpha_i-1}}\log p_i$.
	Rearranging and using $|h(E_0)-h(E_i)| \le 1/2 \log N$ (e.g.~\cite[Corollaire 2.1.4, page 207]{raynaud})
	we obtain
	\begin{align*}
		\frac{n_1}{\psi(N)} h(E_1)
		&\ge h(E_0) + \frac 12 \log N - \lambda_N - \sum_{j} \frac{n_j}{2\psi(N)}\log N - \sum_{j} \frac{n_j}{\psi(N)}h(E_0)\\
		&= h(E_0) + \frac 12 \log N - \lambda_N - \frac{\psi(N)-n_1}{2\psi(N)}\log N - \frac{\psi(N)-n_1}{\psi(N)}h(E_0)	\\
		&= \frac{n_1}{\psi(N)}h(E_0) + \frac{n_1}{2\psi(N)}\log N - \lambda_N,
	\end{align*}
	where we have grouped curves into $\Gal(\bar\Q/K)$--orbits, each of size $n_j$.
	The number $n_1$ is the number of elliptic curves up to $\bar\Q$--isomorphism that
	are in the $\Gal(\bar\Q/K)$--orbit of $E_1$.
	This implies
	\begin{displaymath}
		h(E_1) \ge h(E_0) + \frac 12\log N - \frac{\psi(N)}{n_1}\lambda_N.
	\end{displaymath}
	We have
	\begin{displaymath}
		\frac{p_i^{\alpha_i} - 1}{(p_i^2-1)p_i^{\alpha_i-1}} \le \frac{p_i^{\alpha_i}}{(p_i^2-1)p_i^{\alpha_i-1}}
		= \frac{p_i}{p_i^2-1} \le \frac 4{3p_i}.
	\end{displaymath}
	It follows from the last lemma and Lemma \ref{lem:serre_index_bound} that
	\begin{displaymath}
		h(E_1) \ge h(E_0) + \frac 12\log N - \frac{21}3\left[\GL_2(\Z/N\Z):\rho_N(G_K)\right]\log\log N,
	\end{displaymath}
	as desired.
\end{proof}

\begin{cor}\label{cor:lowerbound}
	In the setting of the previous proposition,
	let $j_0$ and $j$ be the $j$--invariants of $E_0$ and $E$, respectively.
	We have
	\begin{align*}
		h(j_0) &- 6\log(1+h(j_0))
		+ 6\log N - 84\left[\GL_2(\Z/N\Z):\rho_N(G_K)\right]\log\log N	\\
		&\le h(j) + 16.212
	\end{align*}
\end{cor}

\begin{proof}
	Compare the proof to Proposition 2.1 in \cite{silverman86heightsEC}.
	Using Proposition 3.2 of \cite{loebrich} in the first step and Lemme 7.8 of \cite{gaudronremond}
	on the third we obtain
	\begin{align*}
		\frac 1{12}h(j_0) &- \frac 12\log(1+h(j_0)) - 2.071
		+\frac 12\log N - 7\left[\GL_2(\Z/N\Z):\rho_N(G_K)\right]\log\log N	\\
		&\le h(E_0) +\frac 12\log N - 7\left[\GL_2(\Z/N\Z):\rho_N(G_K)\right]\log\log N
			+ \frac 12 \log\pi	\\
		&\le h(E) + \frac 12 \log\pi	\\
		&\le \frac 1{12}h(j) - \frac 12\log(1+h(j)) -0.72	\\
		&\le \frac 1{12}h(j) - 0.72.
	\end{align*}
	Note that the authors of both cited papers use the normalization of the Faltings height of Deligne.
	Multiplying the inequality by $12$ and rearranging the terms yields the desired inequality.
\end{proof}

In the proof of the next lemma we will use the function
\begin{displaymath}
	D(z) = \max\{1,\lvert \Re(z)\rvert, \Im(z)^{-1} \},	\qquad \text{ for all } z \in \H.
\end{displaymath}
It appears in \cite{habBeyAO}. Note that if $z$ is in $\bar\F$, then $D(z) \le 2/\sqrt{3}$.
The height of an element in $\Mat_2(\Q)$ will be the height of that element when
regarded as a member of $\Q^4$.
The following result can be found in a slightly weaker form in \cite{habBeyAO}.

\begin{lem}\label{lem:ht_rhoz}
	If $z \in \H\setminus\sl2z\zeta$, then for any $\rho \in \sl2z$ with $\rho z\in \F$
	we have $H(\rho) \le 264 D(z)^9$. If $z \in \sl2z\zeta$, then we have $H(\rho)\le 1056D(z)^9$
	for any $\rho \in \sl2z$ with $\rho z \in \F$.
\end{lem}

\begin{proof}
	The result follows from the theory in Chapter 2 of \cite{DiamondShurman05}.
	The result of Habegger and Pila states, that $H(\rho) \le 264D(z)^9$ for some $\rho \in \sl2z$.
	Note that $H(\rho) = H(-\rho)$.
	If $\rho' \in \sl2z$ satisfies $\rho z = \rho'z$, then $\rho^{-1}\rho'$ lies
	in the isotropy subgroup $\sl2z_z$ of $z$.
	Assume $z$ is neither in the $\sl2z$--orbit of $i$ nor of $\zeta$.
	Then since $z$ is not an elliptic point
	we have $\{\pm I\}\sl2z_z = \{\pm I\}$ with $I$ the $2$--by--$2$ identity matrix.
	Thus we have $H(\rho) \le 264D(z)^9$ in any case.

	Assume $z \in \sl2z i$. Then $\rho^{-1}\rho' \in \sl2z_i = \langle S \rangle$
	with $S = \begin{smatrix}0&-1\\1&0\end{smatrix}$. But then $H(\rho) = H(\rho')$
	and the result follows.

	Finally, assume $z \in \sl2z \zeta$.
	Then $\rho^{-1}\rho' \in \sl2z_\zeta = \langle \begin{smatrix}0&-1\\1&1\end{smatrix}\rangle$,
	so that $H(\rho') \le 4H(\rho) \le 1056D(z)^9$.
\end{proof}

Let $N, m, n, l$ be integers satisfying $1 < N = mn$ and $0 \le l < n$. Then
we have $\lvert \Re\left(\frac{m\tau + l}n\right) \rvert \le N \left(\lvert\Re(\tau)\rvert + 1\right) \le N (D(\tau)+1)$,
and we can similarly bound the inverse of the imaginary part by $N (D(\tau) + 1)$.
Thus
\begin{equation}\label{eq:fd_dist}
	D\left(\frac{m\tau + l}n\right) \le N (D(\tau) + 1).
\end{equation}
We will use this for the proof of Lemma \ref{lem:linform}, but first

\begin{lem}\label{lem:tau_less_htj}
	We have $|\tau_0^\sigma|/[K:\Q] \le 3 \max\{1, h(j_0)\}$.
\end{lem}

\begin{proof}
	Put $D = [K:\Q]$. We have
	\begin{align*}
		|\tau_0^\sigma| \le \frac 32 \log \max\{e,|j(\tau_0^\sigma)|\}
	\end{align*}
	by Lemme 1 item (iv) in \cite{faisant1987quelques}. Thus
	\begin{align*}
		\frac{|\tau_0^\sigma|}D &\le \frac 32 \frac 1D \log \max\{e,|j(\tau_0^\sigma)|\}	\\
		&\le \frac 32 \frac 1D \left(1 + \log \max\{1,|j(\tau_0^\sigma)|\}\right)	\\
		&= \frac 32 \frac 1D \left(1 + \log \max\{1,|\sigma(j(\tau_0))|\}\right)	\\
		&\le \frac 32 \frac 1D \left(1 + \sum_{\nu \in M_K} d_\nu \log \max\{1,|j(\tau_0)|_\nu\}\right)	\\
		&\le \frac 32 + \frac 32 h(j_0).
	\end{align*}
	This gives the desired inequality.
\end{proof}

\begin{lem}\label{lem:linform}
	Let $E_0\colon y^2 = 4x^3 - g_2x - g_3$ be the Weierstrass form of an elliptic curve without complex multiplication
	defined over a number field $K$.
	Let $j_0$ be the $j$--invariant of $E_0$ and put $h = \max\{1, h(1,g_2,g_3), h(j_0)\}$.
	Let $\omega_1$ and $\omega_2$ be periods of the elliptic curve
	such that $\tau_0 = \omega_2/\omega_1$ is in $\F$.
	Suppose that $\xi$ is an algebraic number of degree $2$.
	Let $N, m, n, l$ be integers satisfying
	\begin{displaymath}
		N \ge \left({\max\{e^{18\pi h}, [K:\Q], {(4\cdot 10^{11} H(\xi)})^{20}\}} \right)^{1/20}
		=: \mathcal{N}(E_0,\xi),
	\end{displaymath}
	$N = mn$ and $0 \le l < n$.
	Let $\rho \in \sl2z$ satisfy $\rho\begin{smatrix}m&l\\0&n\end{smatrix}.\tau_0 \in \bar\F$.
	Write $\begin{smatrix}\alpha&\beta\\\gamma&\delta\end{smatrix} = \rho\begin{smatrix}m&l\\0&n\end{smatrix}$.
	Then there exists an explicit constant $c_1' \ge 1$ such that
	\begin{displaymath}
		\log \lvert (\alpha-\xi\gamma)\omega_2 + (\beta-\xi\delta)\omega_1 \rvert \ge -c_1' \cdot (\log N)^4.
	\end{displaymath}
	The constant $c_1'$ depends on the elliptic curve $E_0$.
\end{lem}

\begin{proof}
	This is a special case of Th\'eor\`eme 2.1 in \cite{davidlinearforms}.
	We set $D = [K:\Q]$.
	Also see \cite{davidhiratakohno} for a similar result with a computable constant.
	We put $\mathcal{L}(z_0,z_1,z_2) = (\alpha-\xi\gamma) z_1 + (\beta-\xi\delta)z_2$.
	Our elliptic curve and the coefficients are in a number field of degree at most $2D$
	since $\xi$ is quadratic.
	Note that $(\alpha-\xi\gamma) \omega_2 + (\beta-\xi\delta)\omega_1 \not= 0$ otherwise we would have
	\begin{displaymath}
		\tau_0 = \frac{\xi\delta-\beta}{\alpha-\xi\gamma} = \begin{pmatrix}\delta&-\beta\\-\gamma&\alpha\end{pmatrix}.\xi
	\end{displaymath}
	i.e.~there is a isogeny of degree $N$ between elliptic curves with $j$--invariant $j(\tau_0)$
	and $j(\xi)$.
	But one has complex multiplication and the other does not, so this is impossible.
	We choose the variables $u_1, u_2$ in the theorem to be $\omega_2$ and $\omega_1$, respectively.
	Then $\gamma_1 = \gamma_2 = (0,0,1)$ and $v = (1,\omega_2,\omega_1)$.
	We have to estimate the height of the coefficients of the linear form.
	For this, let $H$ denote the multiplicative height.
	Let $\rho = \begin{smatrix}a&b\\c&d\end{smatrix}$. Then $\alpha = ma$ and $\gamma = cm$,
	and by \ref{eq:ht_sum_lower_bound} we obtain
	\begin{align*}
		H(\alpha-\xi \gamma) &\le 2H(\alpha)H(-\xi\gamma)	\\
			&\le 2H(\alpha)H(\xi)H(\gamma)	\\
			&= 2H(am)H(\xi)H(cm)	\\
			&\le 2H(a)H(m)H(\xi)H(c)H(m)	\\
			&= 2H(a)H(\xi)H(c)m^2.
	\end{align*}
	Now $m \le N$ and $H(a),H(c) \le H(\rho)$ so that
	\begin{displaymath}
		H(\alpha-\xi \gamma) \le 2H(\rho)^2H(\xi)N^2.
	\end{displaymath}
	Note that $\begin{smatrix}m&l\\0&n\end{smatrix}.\tau_0$ does not have CM and
	is thus not an elliptic point for $\sl2z$. This means that if $\rho' \in \sl2z$
	transfers the point to the same points as $\rho$ does, then $\rho' = \pm \rho$.
	Since $H(\rho) = H(-\rho)$
	we can use Lemma \ref{lem:ht_rhoz} together with \eqref{eq:fd_dist} to obtain
	\begin{equation}\label{eq:ht_rho}
		H(\rho) \le 264 D\left(\frac{m\tau_0+l}n\right)^9 \le 264 (D(\tau_0)+1)^9 N^9
		\le 2\cdot 10^5 N^9,
	\end{equation}
	because $\tau_0 \in \F$.
	Altogether we have
	\begin{displaymath}
		H(\alpha-\xi\gamma) \le 8\cdot 10^{10} H(\xi) N^{20}.
	\end{displaymath}
	We have $\beta = al+bn$ and $\delta = cl+dn$. Recall $0 \le l < n \le N$ and thus
	\begin{displaymath}
		H(\beta) = |al+bn| \le |al|+|bn| \le (|a|+|b|)N \le 2H(\rho) N
	\end{displaymath}
	and
	\begin{displaymath}
		H(\delta) = |cl+dn| \le |cl|+|dn| \le (|c|+|d|)N \le 2H(\rho) N.
	\end{displaymath}
	For the height of $H(\beta-\xi\delta)$ we obtain
	\begin{align*}
		H(\beta-\xi \delta)
			&\le 2H(\beta)H(\xi)H(\delta)	\\
			&\le 8H(\rho)^2H(\xi)N^2.
	\end{align*}
	Using \eqref{eq:ht_rho} again this yields
	\begin{displaymath}
		H(\beta-\xi \delta) \le 4\cdot 10^{11} H(\xi) N^{20}.
	\end{displaymath}
	Put $V_1 = V_2 = e^{18\pi h}$.
	We have
	\begin{align*}
		\frac{3\pi |u_1|^2}{|\omega_1|^2 \Im(\tau_0) 2D}
		\le \frac{3\pi |u_1|^2}{|\omega_1|^2 \Im(\tau_0) D}
		\le \frac{3\pi |\tau_0|^2}{\Im(\tau_0) D}
		\le \frac{6\pi}D |\tau_0|
		\le 18\pi h
	\end{align*}
	where we used the previous lemma on the last inequality
	and we also have
	\begin{align*}
		\frac{3\pi |u_2|^2}{|\omega_1|^2 \Im(\tau_0) D}
		\le \frac{3\pi}{\Im(\tau_0) D}
		\le 6\pi
	\end{align*}
	since $|\tau_0|^2/\Im(\tau_0) \le 2|\tau_0|/\sqrt{3}$ and $\Im(\tau_0) \ge \sqrt{3}/2$ for $\tau_0 \in \bar\F$.
	Therefore, equation (3) of Th\'eor\`eme 2.1 in \cite{davidlinearforms} is satisfied independently
	of whether $\xi$ is in $K$ or not.
	Assume
	\begin{displaymath}
		N \ge \left({\max\{e^{18 \pi h}, D, {(4\cdot 10^{11} H(\xi)})^{20}\}} \right)^{1/20}.
	\end{displaymath}
	Define
	\begin{displaymath}
		B = N^{21}.
	\end{displaymath}
	We picked $N$ large enough so that
	\begin{displaymath}
		B \ge \max\{e^{18\pi h},H(\alpha-\xi\gamma), H(\beta-\xi\delta)\}.
	\end{displaymath}
	This implies $B \ge V_1^{1/D} = V_2^{1/D}$. Thus, equations (1) and (2) of
	the theorem in \cite{davidlinearforms} are satisfied, and
	we are in the situation of the theorem to obtain as a result the lower bound
	\begin{align*}
		\log \lvert \mathcal{L}(v) \rvert &\ge -C \cdot 2^6\cdot D^6 (\log B + \log(2D))\cdot(\log\log B+h+\log(2D))^3\log V_1 \log V_2	\\
		&\ge -C \cdot 2^6 \cdot D^6 \cdot 54\cdot(18\pi h)^2 \cdot (\log B)^4
	\end{align*}
	since $h \le \log B$ and $\log(2D) \le \log B$.
	If we substitute $B$ and take $C$ from \cite{davidlinearforms} we get
	\begin{align*}
		\log \lvert \mathcal{L}(v) \rvert
		&\ge -C \cdot 2^6 \cdot D^6 \cdot 54\cdot(18\pi h)^2 \cdot 21^4\cdot (\log N)^4	\\
		&\ge -10^{54} \cdot D^6 \cdot h^2 \cdot (\log N)^4.
	\end{align*}
	This gives the desired inequality of the lemma.
\end{proof}

We use the definitions
\begin{displaymath}
	\F_+ = \{ \tau \in \F; 0 \le \Re(\tau) \le 1/2 \}
\end{displaymath}
and
\begin{displaymath}
	\F_- = \{ \tau \in \F; -1/2 \le \Re(\tau) \le 0 \}.
\end{displaymath}
The following lemma can be found in \cite{bilulucapizarro}.

\begin{lem}\label{lem:jbound}
	For $\tau \in \F_+$ we either have $|\tau-\zeta| \ge 10^{-3}$ and $|j(\tau)| \ge 4.4\cdot 10^{-5}$
	or $|\tau-\zeta| \le 10^{-3}$ and
	\begin{displaymath}
		 44000|\tau-\zeta|^3 \le |j(\tau)| \le 47000|\tau-\zeta|^3.
	\end{displaymath}
	For $\tau \in \F_-$ we either have $|\tau-\zeta^2| \ge 10^{-3}$ and $|j(\tau)| \ge 4.4\cdot 10^{-5}$
	or $|\tau-\zeta^2| \le 10^{-3}$ and
	\begin{displaymath}
		 44000|\tau-\zeta|^3 \le |j(\tau)| \le 47000|\tau-\zeta|^3.
	\end{displaymath}
\end{lem}

We fix $E_0$ given by a Weierstrass equation $y^2 = 4x^3 - g_2x - g_3$, and assume
it is defined over a number field $K$. Let $j_0$ be its $j$--invariant and pick
$\tau_0 \in \F$ with $j(\tau_0) = j_0$.
Let $E$ be an elliptic curve with $j$--invariant $j$ that is $N$--isogenous to $E_0$.
As before, we set $j(\tau_0^\sigma) = \sigma(j(\tau_0))$ with $\tau^\sigma_0 \in \F$
for any field embedding $\sigma\colon K \rightarrow \C$. By $E_0^\sigma$ and $E^\sigma$
we denote the Galois conjugates of $E_0$ and $E$, respectively.

\begin{lem}\label{lem:conjbound}
	Let $N \ge \mathcal{N}(E_0^\sigma, \zeta)$.
	We have $\log \lvert \sigma(j) \rvert \ge -c_1 \cdot (\log N)^6 - c_2$ for
	any $\Q$--homomorphism $\sigma\colon K \rightarrow \C$,
	where the constants are explicit and only depend on the fixed elliptic curve $E_0$.
	We have $c_1 \ge 1$ and we can have $c_2 \ge 0$.
\end{lem}

\begin{proof}
	We assume $|\sigma(j)| \le 10^{-3}$ for now.
	We have an $N$--isogeny between $E_0^\sigma$ and $E^\sigma$ since $E_0$ and $E$ are $N$--isogenous.
	Let $E_0^\sigma(\C) \simeq \C/(\omega_{0,1}^\sigma \Z + \omega_{0,2}^\sigma \Z)$
	with $\tau_0^\sigma = \omega_{0,2}^\sigma/\omega_{0,1}^\sigma$
	in the fundamental domain.
	Similarly, let $\tau^\sigma$ correspond to $E^\sigma(\C)$.
	We can choose $ \omega_{1}^\sigma$ and $\omega_{2}^\sigma$ such
	that $\tau^\sigma = \rho \begin{smatrix}m&l\\0&n\end{smatrix}\tau_0^\sigma$ and
	such that $\tau^\sigma$ is in the fundamental domain $\F$.
	Write $\begin{smatrix}\alpha&\beta\\\gamma&\delta\end{smatrix} = \rho\begin{smatrix}m&l\\0&n\end{smatrix}$.
	A similar estimate to the one in the proof of Lemma \ref{lem:linform} shows
	\begin{align}\label{eq:scaledfp}
		\begin{split}
			\lvert \gamma\omega_{0,2}^\sigma + \delta\omega_{0,1}^\sigma \rvert
			&\le 3N \max\{|\omega_{0,1}^\sigma|,|\omega_{0,2}^\sigma|\} H(\rho)	\\
			&\le 792 (D(\tau_0^\sigma)+1)^9 \max\{|\omega_{0,1}^\sigma|,|\omega_{0,2}^\sigma|\} N^{10}	\\
			&\le 10^6 \max\{|\omega_{0,1}^\sigma|,|\omega_{0,2}^\sigma|\} N^{10},
		\end{split}
	\end{align}
	since $D(\tau_0^\sigma) \le 2/\sqrt{3}$.
	Note that we have $\tau^\sigma \not= \zeta$ since $E$ does not have CM. 
	We have
	\begin{equation}\label{eq:linlog_trans}
		\begin{split}
			\log \lvert \tau^\sigma - \zeta \rvert
				&= \log \left\lvert \begin{smatrix}\alpha&\beta\\\gamma&\delta\end{smatrix}.\tau_0^\sigma - \zeta \right\rvert	\\
			&= \log \left\lvert \frac{\alpha\tau_0^\sigma + \beta}{\gamma\tau_0^\sigma + \delta} - \zeta \right\rvert	\\
			&= \log \left\lvert \frac{\alpha\omega_{0,2}^\sigma + \beta\omega_{0,1}^\sigma}{\gamma\omega_{0,2}^\sigma + \delta\omega_{0,1}^\sigma}
					- \zeta \right\rvert	\\
			&= \log\left( \frac 1{\lvert\gamma\omega_{0,2}^\sigma + \delta\omega_{0,1}^\sigma\rvert}
				\lvert \alpha\omega_{0,2}^\sigma + \beta\omega_{0,1}^\sigma
					- \zeta(\gamma\omega_{0,2}^\sigma + \delta\omega_{0,1}^\sigma) \rvert \right)	\\
			&= -\log \lvert\gamma\omega_{0,2}^\sigma + \delta\omega_{0,1}^\sigma\rvert
				+ \log\lvert (\alpha-\zeta\gamma)\omega_{0,2}^\sigma + (\beta-\zeta\delta)\omega_{0,1}^\sigma \rvert.
		\end{split}
	\end{equation}
	We can use \eqref{eq:scaledfp} in the first step and Lemma \ref{lem:linform} the second to get
	\begin{align*}
		\log \lvert \tau^\sigma - \zeta \rvert
		&\ge -\log\left(10^6 \max\{|\omega_{0,1}^\sigma|,|\omega_{0,2}^\sigma|\} N^{10}\right)
			+ \log\lvert (\alpha-\zeta\gamma)\omega_{0,2}^\sigma + (\beta-\zeta\delta)\omega_{0,1}^\sigma \rvert	\\
		&\ge -\log\left(10^6 \max\{|\omega_{0,1}^\sigma|,|\omega_{0,2}^\sigma|\} N^{10}\right)
			- c_1' \cdot (\log N)^6,
	\end{align*}
	where $c_1'$ is the constant from Lemma \ref{lem:linform}.
	The same bound holds for $\zeta$ replaced by $\zeta^2$ since $\mathcal{N}(E,\zeta) = \mathcal{N}(E,\zeta^2)$.
	Assuming that $\tau^\sigma$ is closer to $\zeta$, Lemma \ref{lem:jbound} says
	\begin{displaymath}
		\lvert \sigma(j) \rvert = \lvert j(\tau^\sigma) \rvert
		\ge 44000|\tau^\sigma-\zeta|^3.
	\end{displaymath}
	This implies
	\begin{align*}
		\log\lvert \sigma(j) \rvert =& \log\lvert j(\tau^\sigma) \rvert
			\ge \log 44000 + \log|\tau^\sigma-\zeta|^3	\\
		\ge& \log 44000 -3\log\left( 10^6\max\{|\omega_{0,1}^\sigma|,|\omega_{0,2}^\sigma|\} \right)	\\
			&-10 \log N - 3c_1' \cdot (\log N)^6	\\
		\ge& -14 -3\log\left( \max\{|\omega_{0,1}^\sigma|,|\omega_{0,2}^\sigma|\} \right)	\\
			&- 4\cdot 10^{54} \cdot D^6 \cdot h^2\cdot (\log N)^6.
	\end{align*}
	So we can put $c_1 = 2\cdot 10^{51} \cdot D^6 \cdot h^2 \ge 1$. Since
	$|\omega_{0,1}^\sigma|$ and $|\omega_{0,2}^\sigma|$ can be small we put
	$c_2 = 14 +3\log\left( \max\{1,|\omega_{0,1}^\sigma|,|\omega_{0,2}^\sigma|\} \right)$.

	If $|\sigma(j)| \ge 10^{-3}$, then
	\begin{displaymath}
		\log(|\sigma(j)|) \ge \log(10^{-3}) \ge -7 > -c_2,
	\end{displaymath}
	so the bound is true.
\end{proof}

\begin{prop}\label{prop:j_isogeny_bound}
	Let $j_0$ and $j$ be $j$--invariants of elliptic curves, where $j_0$
	is associated to the elliptic curve $E_0/K$
	given by $E_0\colon y^2 = 4x^3 - g_2x -g_3$.
	Put $h = \max\{1, h(1,g_2,g_3), h(j_0)\}$ and $j(\tau_0) = j_0$ with $\tau_0 \in \F$.
	Assume we have a cyclic isogeny of degree $N$ between $E_0$ and an
	elliptic curve corresponding to $j$.
	Further assume that $j$ is an algebraic unit.
	If
	\begin{displaymath}
		N \ge \max\left\{ 4\cdot 10^{11}, [K:\Q], e^{18\pi h} \right\},
	\end{displaymath}
	then the height of $j$ can be estimated by
	\begin{align*}
		h(j)
		\le& {6\cdot 10^7 h [K:\Q] [\GL_2(\Z/N\Z):\rho_N(G_{K})]}\left(N^{-1/10}
					+ \sqrt\varepsilon\right) \left(c_1 (\log N)^6 + c_2\right)	\\
					&+ 3\lvert \log \varepsilon \rvert
	\end{align*}
	where $0 < \varepsilon < 10^{-5}$ is arbitrary.
\end{prop}

\begin{proof}
	Let $E$ be an elliptic curve corresponding to $j$, so
	that there is a cyclic isogeny of degree $N$ between $E_0$ and $E$.
	Let $\Phi \subset E_0[N]$ be the kernel of the given isogeny $E_0 \rightarrow E$.
	Put $G = \{ \sigma \in \Gal(K(E_0[N])/K); \sigma(\Phi) = \Phi \}$, and let
	$K^\Phi = \{ \alpha \in K(E_0[N]); \sigma(\alpha) = \alpha \text{ for all } \sigma \in G\}$ be the
	fixed field of $\Phi$. By basic Galois theory $\Gal(K(E_0[N])/K^\Phi)$ is
	equal to $G$, and hence
	$\sigma(\Phi) = \Phi$ for all $\sigma \in \Gal(\overline{K^\Phi}/K^\Phi)$.
	This implies
	\begin{align*}
		B &= \left|\left\{ \sigma(\Phi) : \sigma \in \Gal(K(E_0[N])/K)\right\}\right|	\\
		&= \frac{|\Gal(K(E_0[N])/K)|}{|G|}	\\
		&= |\Gal(K^\Phi/K)| = [K^\Phi:K].
	\end{align*}

	By Remark III.4.13.2 in \cite{silverman86AEC} the elliptic curve $E$ is defined over $K^\Phi$,
	and hence $j \in K^\Phi$.
	Let $D$ be the degree of $K^\Phi$ over $\Q$.
	Put $\varepsilon_0 = 44000\varepsilon^3$.
	By \eqref{eq:alg_int_height} we have
	\begin{align}\label{eq:ht_eps0}
		\notag
		h(j) &= -\frac 1D \left( \sum_{\lvert \sigma(j) \rvert < \varepsilon_0} \log \lvert \sigma(j) \rvert
					+ \sum_{\varepsilon_0 \le \lvert \sigma(j) \rvert < 1} \log \lvert \sigma(j) \rvert \right)	\\
		&\le -\frac 1D  \sum_{\lvert \sigma(j) \rvert < \varepsilon_0} \log \lvert \sigma(j) \rvert
					+ \lvert \log \varepsilon_0 \rvert.
	\end{align}
	Recall the definitions
	\begin{displaymath}
		\Gamma_\varepsilon = \left\{ \sigma\colon K \rightarrow \C; \tau^\sigma \in \Sigma_\varepsilon \right\}
	\end{displaymath}
	with
	\begin{displaymath}
		\Sigma_\varepsilon = \left\{ \tau \in \F; \lvert j(\tau) \rvert < \varepsilon \right\}.
	\end{displaymath}
	If $|\sigma(j)| = |j(\tau^\sigma)| < \varepsilon_0 \le 10^{-3}$ and $\tau^\sigma \in \F_+$,
	then by Lemma \ref{lem:jbound}
	\begin{equation}\label{eq:close_pts}
		|\tau^\sigma-\zeta|^3 \le \frac{|j(\tau^\sigma)|}{44000} < \frac{\varepsilon_0}{44000} = \varepsilon^3,
	\end{equation}
	i.e.~$|\tau^\sigma-\zeta| < \varepsilon$. If $\tau^\sigma$ is not in $\F_+$ but in $\F_-$, then
	$|\tau^\sigma-\zeta^2| < \varepsilon$ also follows from $|\sigma(j)| < \varepsilon_0$ and Lemma \ref{lem:jbound}.
	We have $\sigma \in \Gamma_{\varepsilon}$ since
	\begin{displaymath}
		|\sigma(j)| = |j(\tau^\sigma)| < \varepsilon_0 =  47000\varepsilon^3 \le 47000\cdot 10^{-10}\varepsilon \le \varepsilon.
	\end{displaymath}
	Continuing the estimate of \eqref{eq:ht_eps0} this gives
	\begin{align}\label{eq:j_ht_eps}
		\notag
		h(j) &\le -\frac 1D  \sum_{\lvert \sigma(j) \rvert < \varepsilon_0} \log \lvert \sigma(j) \rvert
					+ \lvert \log \varepsilon_0 \rvert	\\
		\notag
		&\le \frac{\#\Gamma_{\varepsilon_0}}D \max_{|\sigma(j)|<\varepsilon_0}\left\{\log \left\lvert \sigma(j)^{-1} \right\rvert\right\}
					+ 3\lvert \log \varepsilon \rvert - \log 44000	\\
		&\le \frac{\#\Gamma_{\varepsilon_0}}D \max_{|\sigma(j)|<\varepsilon_0}\left\{\log \left\lvert \sigma(j)^{-1} \right\rvert\right\}
					+ 3\lvert \log \varepsilon \rvert.
	\end{align}
	Since $\varepsilon \le 10^{-5} < 3/200^2$ we can apply Proposition \ref{prop:clusterbound} to each
	pair $(E_0^\sigma, \zeta)$ and $(E_0^\sigma, \zeta^2)$ where $\sigma$ runs over all embeddings
	$\sigma\colon K \hookrightarrow \C$ as follows.
	For each $\sigma \in \Gamma_{\varepsilon_0}$ the number $\tau^\sigma$ is close to either $\zeta$ or $\zeta^2$
	as we have seen in \eqref{eq:close_pts} and gives an $N$--isogeny from $E_0^\sigma$ to $E^\sigma$.
	Thus we can bound $\#\Gamma_{\varepsilon_0}$ by
	\begin{displaymath}
		\#\Gamma_{\varepsilon_0} \le 6\cdot 10^7 h [K:\Q]^2 \left(\sqrt{N}\sigma_0(N) + \sqrt\varepsilon\psi(N)\right).
	\end{displaymath}
	after applying Lemma \ref{lem:tau_less_htj}.
	We also have $\varepsilon_0 \le 10^{-3}$ and $N \ge \mathcal{N}(E_0^\sigma, \zeta)$ by assumption
	and the previous lemma,
	so we can apply Lemma \ref{lem:conjbound} to get
	\begin{displaymath}
		\max_{|\sigma(j)|<\varepsilon_0}\left\{\log \left\lvert \sigma(j)^{-1}\right\lvert\right\}
		\le c_1(\log N)^6 + c_2.
	\end{displaymath}
	Using the last two inequalities for \eqref{eq:j_ht_eps} we obtain
	\begin{align*}
		h(j)
		\le \frac{6\cdot 10^7 h [K:\Q] \left(\sqrt{N}\sigma_0(N) + \sqrt\varepsilon\psi(N)\right)}{B} \left(c_1 (\log N)^6 + c_2\right)
					+ 3\lvert \log \varepsilon \rvert
	\end{align*}
	since we have $D = [K^\Phi:\Q] = B\cdot [K:\Q]$.
	Nicolas shows on page 229 in  \cite{nicolas1987hcn} that $\sigma_0(N) \le N^{2/5}$ for $N \ge 10^7$.
	Moreover, we have
	\begin{align*}
		\sqrt{N}\sigma_0(N) \cdot B^{-1} \le N^{9/10}\cdot B^{-1} &\le N^{9/10} \cdot \psi(N)^{-1} \cdot [\GL_2(\Z/N\Z):\rho_N(G_K)]	\\
		&\le N^{-1/10} \cdot [\GL_2(\Z/N\Z):\rho_N(G_K)].
	\end{align*}
	by Lemma \ref{lem:serre_index_bound}.
	Using this in the inequality for the height above, and Lemma \ref{lem:serre_index_bound} again, we get
	\begin{align*}
		h(j)
		\le& {6\cdot 10^7\cdot h[K:\Q] [\GL_2(\Z/N\Z):\rho_N(G_{K})]}\left(N^{-1/10}
					+ \sqrt\varepsilon\right) \left(c_1 (\log N)^6 + c_2\right)	\\
					&+ 3\lvert \log \varepsilon \rvert,
	\end{align*}
	as desired.
\end{proof}

\begin{theorem}\label{thm:noncm_finitude}
	
	Let $j_0$ be the $j$--invariant of an elliptic curve without complex multiplication.
	Then there are at most finitely many $j$--invariants $j$ of elliptic curves that are isogenous to
	an elliptic curve corresponding to $j_0$ and such that $j$ is an algebraic unit.

\end{theorem}

\begin{proof}
	Let $E_0$ and $E$ be elliptic curves with $j$--invariants $j_0$ and $j$, respectively. Suppose that there is
	an isogeny of degree $N$ between them. We may assume that $N$ is minimal. By Lemma 6.2 in \cite{masserwustholz}
	the isogeny is cyclic. If $N$ is large enough, then Corollary \ref{cor:lowerbound} gives a lower bound for
	the height of $j$
	\begin{equation}\label{eq:iso_lowerbound}
		\begin{split}
			h(j) \ge h(j_0) &- 6\log(1+h(j_0)) + 6\log N \\&- 84\left[\GL_2(\Z/N\Z):\rho_N(G_K)\right]\log\log N
			- 16.212
		\end{split}
	\end{equation}
	Moreover, if $j$ is an algebraic unit and $N$ is large as in Proposition \ref{prop:j_isogeny_bound},
	that proposition yields the upper bound
	\begin{align*}
		h(j)
		\le& {6\cdot 10^7\cdot h[K:\Q] [\GL_2(\Z/N\Z):\rho_N(G_{K})]}\left(N^{-1/10}
					+ \sqrt\varepsilon\right) \left(c_1 (\log N)^6 + c_2\right)	\\
					&+ 3\lvert \log \varepsilon \rvert,
	\end{align*}
	For large enough $N$, the preconceived restrictions on $\varepsilon$ are met
	if we take $\varepsilon = 1/(\log N)^{12}$ since $N \ge 10^7$ and thus $\varepsilon < 10^{-5}$.
	Therefore, we have
	\begin{equation}\label{eq:iso_upperbound}
		\begin{split}
			h(j) \le {6\cdot 10^7\cdot h[K:\Q] [\GL_2(\Z/N\Z):\rho_N(G_K)]}&\left(N^{-1/10} + \frac 1{(\log N)^6}\right)	\\
								&\cdot\left(c_1 (\log N)^6 + c_2\right)	\\
							&+ 36\log \log N.
		\end{split}
	\end{equation}
	Recall that Serre proved that
	$[\GL_2(\Z/N\Z):\rho_N(G_K)]$ is uniformly bounded in $N$. Also Lombardo gives an explicit bound in \cite{lombardo}.
	As we have seen in Corollary \ref{cor:lowerbound} the lower bound for $h(j)$ grows as $\log N$
	and the upper bound as $\log \log N$.
	This clearly gives a contradiction
	for large enough $N$, which leaves us with only finitely many $N$, and hence finitely many isogenies.
\end{proof}

The next proposition bounds the number of $j$ satisfying the conditions in the theorem. Note that
the index $[\GL_2(\hat\Z):\rho_\infty(G_K)]$ can be bounded explicitly by the result of Lombardo. See \cite{lombardo}
or page \pageref{eq:lombardo}.

\begin{prop}\label{prop:noncm_N_bound}
	
	Let $E_0\colon y^2 = 4x^3 - g_2x -g_3$ be an elliptic curve without complex multiplication
	defined over a number field $K$ of degree $D$. Let $j_0$ be its $j$--invariant with $j(\tau_0) = j_0$
	and $\tau_0 \in \F$. We choose $\omega_1$ and $\omega_2$ with $\omega_2/\omega_1 = \tau_0$
	and $E_0(\C) \simeq \C/(\omega_1\Z + \omega_2\Z)$ and similarly for $E_0^\sigma$, $\sigma\colon K \hookrightarrow \C$.
	Define $h = \max \{1,h(1,g_2,g_3),h(j_0)\}$. If $j$ is the $j$--invariant of an elliptic curve
	isogenous to $E_0$ such that $j$ is a unit, then the degree of the minimal isogeny
	between $j_0$ and $j$ is bounded by
	\begin{displaymath}
		\max\left\{10^{180}(Cc_1)^{20}, (Cc_2)^{10}, e^{Cc_1+Cc_2+c_3}, e^{120^2[\GL_2(\hat\Z):\rho_\infty(G_K)]^2},
		e^{18\pi h}, D \right\},
	\end{displaymath}
	where the constants are given by
	\begin{align*}
		C &= 6\cdot 10^7\cdot h \cdot  D[\GL_2(\hat\Z):\rho_\infty(G_K)],	\\
		c_1 &= 4\cdot 10^{54} D^6 \cdot h^2 \ge 1,	\\
		c_2 &= 14 +3\log\left( \max_\sigma\{1,|\omega_{0,1}^\sigma|,|\omega_{0,2}^\sigma|\} \right) \text{ and }	\\
		c_3 &= 20 - h(j_0) + 6 \log(1+h(j_0)).
	\end{align*}

\end{prop}

Note that $c_3 < 26$ since $-x+6\log(1+x)$ has a maximum at $5$ and $-5 + 6\log(6) < 6$.

\begin{proof}
	We proceed as in the proof of the theorem.
	The inequalities \eqref{eq:iso_lowerbound} and \eqref{eq:iso_upperbound} in the proof of the theorem give
	\begin{align*}
		6\log N \le& C\left(N^{-1/10} + \frac 1{(\log N)^6}\right)
								\left(c_1 (\log N)^6 + c_2\right)	\\
							&+ 36\log \log N + 84\left[\GL_2(\Z/N\Z):\rho_N(G_K)\right]\log\log N	\\
							&- h(j_0) + 6\log(1+h(j_0)) + 16.212
	\end{align*}
	and thus
	\begin{align}
		\notag
		6 \le& \frac 1{\log N}\left( Cc_1 N^{-1/10} (\log N)^6 + Cc_2 N^{-1/10} + Cc_1 + \frac{Cc_2}{(\log N)^6} \right.	\\
		\notag
			&\left. + (84[\GL_2(\Z/N\Z):\rho_N(G_K)] + 36) \log\log N \vphantom{\frac{c_2}{(\log N)^6}} + c_3 \right)	\\
			\notag
			\le& \left( Cc_1 N^{-1/10} (\log N)^5 + Cc_2 \frac{N^{-1/10}}{\log N}
					+ \frac{Cc_1 + Cc_2 + c_3}{\log N} \right.	\\
				&\left. + 120[\GL_2(\Z/N\Z):\rho_N(G_K)] \frac{\log\log N}{\log N} \vphantom{\frac{c_2}{(\log N)^6}} \right)
				\label{eq:j_bound_N}
	\end{align}
	We are going to bound each term by $1$ individually. This will give a contradiction to the lower bound $6$.
	We will work our way from the back to the front.
	We have $\log\log x < (\log x)^{1/2}$ for all $x \ge 10$. Thus
	\begin{displaymath}
		120[\GL_2(\Z/N\Z):\rho_N(G_K)] \le \log N/\log \log N
	\end{displaymath}
	follows from
	\begin{displaymath}
		120[\GL_2(\Z/N\Z):\rho_N(G_K)] \le (\log N)^{1/2} \le \log N/\log \log N,
	\end{displaymath}
	which is true for all $N \ge e^{(120[\GL_2(\hat\Z):\rho_\infty(G_K)])^2}$.	\\
	The next term is $(Cc_1 + Cc_2 + c_3)/\log N$. This is bounded by $1$ for all $N \ge e^{Cc_1+Cc_2+c_3}$.

	The second term is less than $1$ if $Cc_2 \le N^{1/10}$ is satisfied and $N\ge 3$.
	This is true for all $N \ge \max\{(Cc_2)^{10},3\}$.

	For the first term we need
	\begin{displaymath}
		Cc_1 \le \frac{N^{1/10}}{(\log N)^5}.
	\end{displaymath}
	We have $\log x \le 40 x^{1/100}$ for all $x \ge 10^{45}$. Thus the bound holds if
	\begin{displaymath}
		Cc_1 \le 10^{-9} N^{1/20} = 10^{-9} 40^5 \frac{N^{1/10}}{(40 N^{1/100})^5} \le \frac{N^{1/10}}{(\log N)^5}.
	\end{displaymath}
	This is true for $N \ge 10^{180} (Cc_1)^{20}$.

	All those assumptions on $N$ together with the constraint
	\begin{displaymath}
		N \ge \max\left\{ 4\cdot 10^{11}, [K:\Q], e^{18\pi h} \right\}
	\end{displaymath}
	we made in the previous proposition gives the desired bound.
\end{proof}

This finishes the case $j-\alpha$ when $\alpha$ is zero. In the next section we are going to discuss the
case when $\alpha$ is different from $0$.

\section{Translates}

Fix $\alpha \in \bar\Q$ the $j$--invariant of an elliptic curve with complex multiplication,
and let $j_0$ be the $j$--invariant of an elliptic curve without complex multiplication.
We further assume $\alpha \not= 0$ since this is the case discussed in the last section.
We now want to bound the $j$--invariants $j$ such that
the corresponding elliptic curve is isogenous to the elliptic curve $E_0$, and such that $j-\alpha$
is an algebraic unit. Note that the previous case is a special case of this where $\alpha = 0$.
Let $\xi$ be imaginary quadratic with $j(\xi) = \alpha$. We proceed as before, i.e.~we want to give
lower and upper bounds of $h(j-\alpha)$ that contradict each other.	\\
On the one hand we have
\begin{equation}\label{eq:trivhtbound}
	h(j-\alpha) \ge h(j) - h(\alpha) - \log 2
\end{equation}
as remarked in \eqref{eq:ht_sum_lower_bound}. So if there is a cyclic $N$--isogeny between the curves
corresponding to $j$ and $j_0$, then Corollary \ref{cor:lowerbound} yields
\begin{align*}
	h(j-\alpha) \ge h(j_0) &- 6\log(1+h(j_0)) + 6\log N	\\
	&- 84\left[\GL_2(\Z/N\Z):\rho_N(G_K)\right]\log\log N - 20 - h(\alpha).
\end{align*}

Now we want to bound the height from above. We need a similar statement to the one in Lemma \ref{lem:conjbound}.
First we want to introduce the following constant
\begin{displaymath}
	c(\xi) = \begin{cases}
		|j'(\xi)|\delta/2	&	\text{if } \xi \in (\partial\F_+\cup\partial\F_-)\setminus\{i\}	\\
		|j''(i)|\delta^2/4	&	\text{if } \xi = i	\\
		\min\left\{|\Im(j(\xi))|, |j'(\xi)|\delta/2 \right\}	&	\text{otherwise}
	\end{cases}
\end{displaymath}
where $\delta$ is defined as the minimum of $\frac{A}{12A + 108B}$ and half the distance
of $\xi$ to any geodesic of $\partial\F_+$ not containing $\xi$, $B$ is defined
as $4\cdot 10^5 \max\{1,|j(\xi)|\}$ and $A = |j''(i)|$
if $\xi = i$ and $A = |j'(\xi)|$ otherwise. We also assumed $\xi \not = \zeta, \zeta^2$.
More details can be found in Lemma 3.8 of \cite{cmcase}.

Recall the definition
\begin{displaymath}
	\mathcal{N}(E_0,\xi) := \left({\max\{e^{6\pi|\tau_0|/[K:\Q]}, e^{e\cdot h}, [K:\Q], {(4\cdot 10^{11} H(\xi)})^{20}\}} \right)^{1/20}.
\end{displaymath}

\begin{lem}
	Let $j(\tau)$ be $N$--isogenous to $E_0$ and $N \ge \mathcal{N}(E_0^\sigma, \xi^\sigma)$.
	\begin{displaymath}
		\log|\sigma(j-\alpha)| \ge -c_1 (\log N)^6 - c_2
	\end{displaymath}
	for any embedding $\sigma\colon K \hookrightarrow \C$.
	Here the constants are effective and depend on the fixed elliptic curve $E_0$ and $c_2$
	additionally depends on $\xi$.
	We also have $c_1 \ge 1$ and we can have $c_2 \ge 0$.
\end{lem}

\begin{proof}
	The setup is the same as in Lemma \ref{lem:conjbound}.
	Let $E$ be an elliptic curve with $j$--invariant $j(\tau)$ and $E^\sigma$ be the elliptic curve $E$ conjugated by $\sigma$.
	Then there is an $N$--isogeny between $E_0 = E_0^\sigma$ and $E^\sigma$ since $E_0$ and $E$ are $N$--isogenous.
	Let $E_0^\sigma(\C) \simeq \C/(\omega_{0,1}^\sigma \Z + \omega_{0,2}^\sigma \Z)$
	with $\tau_0^\sigma = \omega_{0,2}^\sigma/\omega_{0,1}^\sigma$
	in the fundamental domain.
	Similarly, let $\tau^\sigma$ correspond to $E^\sigma(\C)$.
	We can choose $ \omega_{1}^\sigma$, $\omega_{2}^\sigma$ and $\rho \in \sl2z$ such
	that $\tau^\sigma = \rho \begin{smatrix}m&l\\0&n\end{smatrix}\tau_0^\sigma$ and
	such that $\tau^\sigma$ is in the fundamental domain $\F$.
	Write $\begin{smatrix}\alpha&\beta\\\gamma&\delta\end{smatrix} = \rho\begin{smatrix}m&l\\0&n\end{smatrix}$.

	Assume $|\sigma(j-\alpha)| < c(\xi^\sigma)$ for a moment.
	Put $A^\sigma = |j''(\xi^\sigma)|$ if $\xi^\sigma = i$ and $A^\sigma = |j'(\xi^\sigma)|$ otherwise.
	By Lemma 3.9 in \cite{cmcase} we obtain $|\tau^\sigma - M\xi^\sigma| < \delta^\sigma$
	for some $M \in \T$ with
	$\T = \{ \begin{smatrix}1&0\\0&1\end{smatrix}, \begin{smatrix}1&\pm 1\\0&1\end{smatrix}, \begin{smatrix}0&-1\\1&0\end{smatrix} \}$.
	The number $\delta^\sigma$ is the $\delta$ stated above but associated to $\xi^\sigma$.
	Since $\delta^\sigma$ satisfies by definition $\delta^\sigma \le \frac{A^\sigma}{12A^\sigma+108B^\sigma}$,
	where $B^\sigma = 4\cdot 10^5 \max\{1,|j(\xi^\sigma)|\}$,
	we obtain by Lemma 3.7 in \cite{cmcase} the inequality
	\begin{equation}\label{eq:j_to_pts}
		|j(\tau^\sigma) - j(\xi^\sigma)| \ge \frac{A^\sigma}4 \lvert \tau^\sigma - M\xi^\sigma \rvert^2
	\end{equation}
	for some $M \in \T$.

	Equation \eqref{eq:scaledfp} says
	$|\gamma\omega_{0,2}^\sigma + \delta\omega_{0,1}^\sigma| \le 10^6 \max\{|\omega_{0,1}^\sigma|,|\omega_{0,2}^\sigma|\} N^{10}$.
	Note that we also have $\tau^\sigma \not= M\xi^\sigma$ since $\xi$ comes from a curve with
	complex multiplication. We can substitute $M\xi^\sigma$ for $\zeta$ in \eqref{eq:linlog_trans}
	to get the equality
	\begin{displaymath}
			\log \lvert \tau^\sigma - M\xi^\sigma \rvert
			= -\log \lvert\gamma\omega_{0,2}^\sigma + \delta\omega_{0,1}^\sigma\rvert
				+ \log\lvert (\alpha-M\xi^\sigma\gamma)\omega_{0,2}^\sigma + (\beta-M\xi^\sigma\delta)\omega_{0,1}^\sigma \rvert.
	\end{displaymath}
	Since $\xi^\sigma$ is algebraic of degree two so is $M\xi^\sigma$. We have
	\begin{displaymath}
		\log|\tau^\sigma - M\xi^\sigma|
		\ge -\log\left(10^6 \max\{|\omega_{0,1}^\sigma|,|\omega_{0,2}^\sigma|\} N^{10}\right)
			- c_1' \cdot (\log N)^6.
	\end{displaymath}
	as in the proof of Lemma \ref{lem:conjbound}. Here $c_1'$ is the constant from Lemma \ref{lem:linform}.
	
	As before we obtain by \eqref{eq:j_to_pts}
	\begin{align*}
		\log\lvert \sigma(j-\alpha) \rvert =& \log\lvert j(\tau^\sigma) - j(\xi^\sigma) \rvert
			\ge \log\frac{A^\sigma}{4} + \log|\tau^\sigma-M\xi^\sigma|^2	\\
		\ge& \log\left(\frac{A^\sigma}{4}\right)
							-2\log\left(10^6\max\{|\omega_{0,1}^\sigma|,|\omega_{0,2}^\sigma|\} \right)	\\
			&-10 \log N - 2c_1' \cdot (\log N)^6	\\
		\ge& \log\left(\frac{A^\sigma}{4}\right)
							-2\log\left(10^6 \max\{|\omega_{0,1}^\sigma|,|\omega_{0,2}^\sigma|\} \right)	\\
			&- 3c_1' \cdot (\log N)^6
	\end{align*}
	where we have used the fact that $N \ge \mathcal{N}(E_0,\xi^\sigma) \ge 4\cdot 10^{11}$.
	If we put
	\begin{displaymath}
		c_2 = \log\max\left\{1, c(\xi^\sigma), \frac 4{A^\sigma}10^{12} |\omega_{0,1}^\sigma|^2,
				\frac 4{A^\sigma}10^{12} |\omega_{0,1}^\sigma|^2\right\}
	\end{displaymath}
	the claim holds independently of whether $|\sigma(j-\alpha)| < c(\xi^\sigma)$ or not.
\end{proof}

We want to apply this lemma.
Recall that $\alpha = j(\xi)$ is a singular modulus. Let $\Delta$ be the discriminant
of the associated endomorphism ring. For any $\sigma\colon K \hookrightarrow \C$ the
singular moduli $j(\xi^\sigma)$ have the same associated discriminant.
The following function can also be found in \cite{cmcase}
\begin{displaymath}
	\Pen(\xi) = \log\max_\sigma\left\{1, c(\xi^\sigma)^{-1}\right\}.
\end{displaymath}

\begin{prop}\label{prop:j-a_N_bound}
	Let $j_0$ and $j$ be $j$--invariants of elliptic curves, where $j_0$
	is associated to the elliptic curve $E_0/K$
	defined by $E_0\colon y^2 = 4x^3 - g_2x -g_3$.
	Put $h = \max\{1, h(1,g_2,g_3), h(j_0)\}$.
	Assume we have a cyclic isogeny of degree $N$ between $E_0$ and an
	elliptic curve corresponding to $j$.
	Further assume that $j$ is an algebraic unit.
	If
	\begin{displaymath}
		N \ge {\max\{e^{18\pi h}, [K:\Q], 4\cdot 10^{11} \sqrt{|\Delta|}\}}
	\end{displaymath}
	then the height of $j$ can be estimated by
	\begin{align*}
		h(j-\alpha)
		\le& \frac{10^8 h [K:\Q]^2 |\Delta|^5 [\GL_2(\Z/N\Z):\rho_N(G_K)]}{[\Q(\alpha):\Q]}(N^{-1/10} + \sqrt{\varepsilon})
																																	\left(c_1 (\log N)^6 + c_2\right)\\
		 &+ \Pen(\xi) +2|\log \varepsilon|.\qedhere
	\end{align*}
	where
	\begin{displaymath}
		0 < \varepsilon < 10^{-4}\min_{\sigma\colon K \hookrightarrow \C}\{|\xi^\sigma|^{-4}\}
	\end{displaymath}
	is arbitrary.
\end{prop}

\begin{proof}
	Recall from the proof of Proposition \ref{prop:j_isogeny_bound} the field $K^\Phi$ for which
	we have $j \in K^\Phi$. We also showed that $[K^\Phi:\Q] = B[K:\Q]$ and $B = [K^\Phi:K]$.
	Let
	\begin{displaymath}
		\varepsilon_0 
		= \varepsilon^2 \cdot \min_{\sigma\colon K \hookrightarrow \C}\left\{1, c(\xi^\sigma)\right\}.
	\end{displaymath}
	Let $|\sigma(j-\alpha)| < \varepsilon_0 < c(\xi^\sigma)$.
	We have $N \ge \mathcal{N}(E_0^\sigma, \xi^\sigma)$ since we have $\sqrt{|\Delta|} \ge H(\xi^\sigma)$
	by Lemma 5 of \cite{habSM} and the statement of Lemma \ref{lem:tau_less_htj}.
	So the previous lemma says
	\begin{displaymath}
		\log|\sigma(j-\alpha)| \ge -c_1 (\log N)^6 - c_2,
	\end{displaymath}
	where we now can take $c_2$ to be the maximum over all constants that we get from the lemma for each $\xi^\sigma$.
	We have $\sigma \in \Gamma(\xi^\sigma, \varepsilon)$ since we assumed $|\sigma(j-\alpha)| < \varepsilon_0 < \varepsilon$.
	The same argument as in the proof of Proposition \ref{prop:j_isogeny_bound} shows
	\begin{align}
		\notag
		h(j-\alpha) &\le -\frac 1D  \sum_{\lvert \sigma(j-\alpha) \rvert < \varepsilon_0} \log \lvert \sigma(j-\alpha) \rvert
					+ \lvert \log \varepsilon_0 \rvert	\\
		\notag
		&\le \frac{\sum_{\sigma\colon K \hookrightarrow \C}\#\Gamma(\xi^\sigma, \varepsilon_0)}D
				\max_{|\sigma(j-\alpha)|<\varepsilon_0}\left\{\log \left\lvert \sigma(j-\alpha)^{-1} \right\rvert\right\}
					+ \lvert \log \varepsilon_0 \rvert	\\
		&\le \frac{\sum_{\sigma\colon K \hookrightarrow \C}\#\Gamma(\xi^\sigma, \varepsilon_0)}D
				\left(c_1 (\log N)^6 + c_2\right)
						+ \lvert \log \varepsilon_0 \rvert,\label{eq:j-a_ht_bound}
	\end{align}
	where $D$ is the degree of $K^\Phi(\alpha)$ over $\Q$.
	Now if $\rho \in \Gamma(\xi^\sigma, \varepsilon_0)$, then $|j(\tau^\rho)-j(\xi^\sigma)|<\varepsilon_0 \le c(\xi^\sigma)$.
	With $\delta^\sigma$ as before we get from Lemma 3.9 in \cite{cmcase}
	\begin{displaymath}
		|\tau^\rho - M\xi^\sigma| < \delta^\sigma
	\end{displaymath}
	for some $M \in \T$.
	As before we put $A^\sigma = |j''(i)|$ if $\xi^\sigma = i$ and $A^\sigma = |j'(\xi)|$ otherwise.
	Recall that $\delta^\sigma \le 1$ so that $c(\xi^\sigma) \le A^\sigma/2$ or $c(\xi^\sigma) \le A^\sigma/4$.
	Lemma 3.7 of \cite{cmcase} then implies
	\begin{displaymath}
		\frac{A^\sigma}{2^k} |\tau^\rho-M\xi^\sigma|^2 \le |j(\tau^\rho)-j(\xi^\sigma)| < \varepsilon_0 \le c(\xi^\sigma) \varepsilon^2 
		\le \frac{A^\sigma}{2^k} \varepsilon^2,
	\end{displaymath}
	where $k \in \{1,2\}$ depending on whether $M\xi^\sigma = i$ or not.
	Therefore we have $|\tau^\rho-M\xi^\sigma| < \varepsilon$. So every $\rho \in \Gamma(\xi^\sigma,\varepsilon_0)$ gives
	a point satisfying $|\tau^\rho - M\xi^\sigma| \le \varepsilon$ and an $N$--isogeny between $E_0^\rho$ and $E^\rho$.
	Note that $M$ can only be different from the identity if $\xi^\sigma$ lies on the boundary of $\F$.
	In any case since $\xi$ (and all $M\xi^\sigma$) is imaginary quadratic, some conjugate lies
	on the imaginary axis and is the largest with respect to the absolute value.
	It is given by $i|\Delta|^{1/2}/2$.
	Moreover, $\varepsilon$ satisfies the conditions of Proposition \ref{prop:clusterbound}.
	We thus can apply Proposition \ref{prop:clusterbound} to bound
	\begin{displaymath}
		\#\Gamma(\xi^{\sigma}, \varepsilon_0) \le 
		4\cdot 10^7[K:\Q]^2 |\Delta|^5 h \left( \sqrt{N} \sigma_0(N) + \sqrt\varepsilon \psi(N) \right).
	\end{displaymath}
	Note that $\varepsilon \le (100^{-1}|M\xi^\sigma|\Im(M\xi^\sigma))^2$ holds since if $M$ is
	different from the identity, then $\xi^\sigma$ lies on the boundary of the fundamental domain
	and we obtain $|M\xi^\sigma| = |\xi^\sigma|$ and $\Im(M\xi^\sigma) = \Im(\xi^\sigma)$.
	We can continue the height estimate in \eqref{eq:j-a_ht_bound}
	\begin{align*}
		h(j-\alpha)
		\le& [K:\Q]\frac{4[K:\Q]^2 10^7 |\Delta|^5 h\left(\sqrt{N}\sigma_0(N) + \sqrt\varepsilon\psi(N)\right)}D
							\left(c_1 (\log N)^6 + c_2\right)	\\
					&+ \lvert \log \varepsilon_0 \rvert	\\
		\le& [K:\Q]\frac{4[K:\Q]^2 10^7 |\Delta|^5 h \left(\sqrt{N}\sigma_0(N) + \sqrt\varepsilon\psi(N)\right)}{[K^\Phi:\Q][\Q(\alpha):\Q]}
																																	\left(c_1 (\log N)^6 + c_2\right)	\\
					&+ \lvert \log \varepsilon_0 \rvert	\\
		\le& \frac{4[K:\Q]^2 10^7 |\Delta|^5 h \left(\sqrt{N}\sigma_0(N) + \sqrt\varepsilon\psi(N)\right)}{B[\Q(\alpha):\Q]}
																																	\left(c_1 (\log N)^6 + c_2\right)
					+ \lvert \log \varepsilon_0 \rvert.
	\end{align*}
	We have bounded the term $\left(\sqrt{N}\sigma_0(N) + \sqrt\varepsilon\psi(N)\right)/B$ in Proposition \ref{prop:j_isogeny_bound},
	so that we obtain
	\begin{align*}
		h(j-\alpha)
		\le& \frac{10^8 [K:\Q]^2 |\Delta|^5 h [\GL_2(\Z/N\Z):\rho_N(G_K)]}{[\Q(\alpha):\Q]}(N^{-1/10} + \sqrt{\varepsilon})
																																	\left(c_1 (\log N)^6 + c_2\right) \\
					&+ \lvert \log \varepsilon_0 \rvert.
	\end{align*}
	Now
	\begin{align*}
		|\log \varepsilon_0|
		= 2|\log \varepsilon| + |\log \min_{\sigma}\left\{1, c(\xi^\sigma)\right\}|
		= 2|\log \varepsilon| + \Pen(\xi).
	\end{align*}
	Replacing this into the height bound be obtain
	\begin{align*}
		h(j-\alpha)
		\le& \frac{10^8 [K:\Q]^2 |\Delta|^5 h[\GL_2(\Z/N\Z):\rho_N(G_K)]}{[\Q(\alpha):\Q]}(N^{-1/10} + \sqrt{\varepsilon})
																																	\left(c_1 (\log N)^6 + c_2\right)\\
		 &+ \Pen(\xi) +2|\log \varepsilon|.\qedhere
	\end{align*}
\end{proof}

\begin{theorem}\label{thm:noncm_ja_finitude}
	
	Assume $\alpha$ is the $j$--invariant of an elliptic curve with CM.
	Let $j_0$ be the $j$--invariant of an elliptic curve without CM.
	Then there are at most finitely many $j$--invariants $j$ of elliptic curves that are isogenous to
	an elliptic curve corresponding to $j_0$ and such that $j-\alpha$ is an algebraic unit.

\end{theorem}

\begin{proof}
	In the same situation as before we get an additional $-\log 2-h(\alpha)$ term from \eqref{eq:trivhtbound}
	for the lower bound and obtain
	\begin{align*}
		h(j_0) &- 6\log(1+h(j_0)) + 6\log N - 84\left[\GL_2(\Z/N\Z):\rho_N(G_K)\right]\log\log N	\\
			&- 20 - h(\alpha) \le h(j-\alpha).
	\end{align*}
	We want to pick $\varepsilon = 1/(\log N)^{12}$ again.
	Thus, $N$ must be large enough so that
	\begin{displaymath}
		\varepsilon = 1/(\log N)^{12} \le 10^{-4} \min_\sigma\{|\xi^\sigma|^{-4}\}
	\end{displaymath}
	or equivalently
	\begin{displaymath}
		(\log N)^{12} \ge 10^{4} \max_\sigma\{|\xi^\sigma|^{4}\}.
	\end{displaymath}
	But as mentioned in the previous proof, $\xi$ and $|\xi^\sigma|$ are imaginary quadratic
	and one of its conjugates is $i|\Delta|/2$ and has maximal modulus amongst them.
	Hence it suffices for $N$ to satisfy $\log N \ge 3 |\Delta|$, i.e.
	\begin{displaymath}
		N \ge e^{3 |\Delta|}.
	\end{displaymath}
	If $N$ additionally satisfies the conditions of the previous proposition then
	\begin{align*}
		h(j-\alpha) \le& \frac{10^8 h [K:\Q]^2 |\Delta|^5 [\GL_2(\Z/N\Z):\rho_N(G_K)]}{[\Q(\alpha):\Q]}	\\
			&\cdot\left(N^{-1/10} + \frac 1{(\log N)^6}\right) \left(c_1 (\log N)^6 + c_2\right)	\\
							&+ \Pen(\xi) + 24\log \log N.
	\end{align*}
	The growth of the bounds for $h(j-\alpha)$ is as before, and we get the same contradiction.
\end{proof}

\noindent In total we obtain the following result. We also recall that $c_3 < 26$.

\begin{prop}\label{prop:noncm_ja_bound}
	Let $E_0\colon y^2 = 4x^3 - g_2x -g_3$ be an elliptic curve without complex multiplication
	defined over a number field $K$ of degree $D$. Let $j_0$ be its $j$--invariant with $j(\tau_0) = j_0$
	and $\tau_0 \in \F$.
	Define $h = \max \{1,h(1,g_2,g_3),h(j_0)\}$.
	Let $\xi \in \F$ be imaginary quadratic and let $\Delta$ be the discriminant of the
	endomorphism ring. Put $\alpha = j(\xi)$.
	If $j$ is the $j$--invariant of an elliptic curve
	isogenous to the elliptic curve $E_0$ and $j-\alpha$ is a unit,
	then the degree of the minimal isogeny is bounded by
	\begin{gather*}
		\max\left\{10^{180}(\hat{C}c_1)^{20}, (\hat{C}c_2)^{10}, e^{\hat{C}c_1+\hat{C}c_2+c_3+\Pen(\xi)},
	 	e^{120^2[\GL_2(\hat\Z):\rho_\infty(G_K)]^2},	\right.	\\
		\left. e^{18\pi h}, [K:\Q], e^{3|\Delta|}, 4\cdot 10^{11}\sqrt{|\Delta|} \right\},
	\end{gather*}
	where $\hat C = 10^8 h [K:\Q]^2|\Delta|^5 [\GL_2(\hat\Z):\rho_\infty(G_K)]/[\Q(\alpha):\Q]$
	and $c_3 = 20 - h(j_0) + 6 \log(1+h(j_0))$.
\end{prop}

\begin{proof}
	The bounds from the previous proof give the same inequality as in \eqref{eq:j_bound_N}
	with $C$ replaced by the new constant $\hat C$ and the third term becomes
	\begin{displaymath}
		\frac{\hat{C}c_1+\hat{C}c_2+c_3 + \Pen\left(\xi \right)}{\log N}.
	\end{displaymath}
	Also we have the additional prerequisites $N \ge e^{3|\Delta|}$ and $N \ge 4\cdot 10^{11} \sqrt{|\Delta|}$
	from the proof of the last theorem.
\end{proof}

\vspace{3cm}
\addcontentsline{toc}{chapter}{\bibname}
\thispagestyle{plain}
\printbibliography

\end{document}